\documentclass{amsart}

\usepackage{color}
\usepackage{amssymb} 

\newtheorem{theorem}{Theorem}[section]
\newtheorem{lemma}[theorem]{Lemma}
\newtheorem{proposition}[theorem]{Proposition}

\theoremstyle{definition}

\theoremstyle{remark}
\newtheorem{remark}[theorem]{Remark}

\numberwithin{equation}{section}

\newcommand{\R}{{\mathbb R}}
\newcommand{\N}{{\mathbb N}}
\newcommand{\Z}{{\mathbb Z}}
\newcommand{\ep}{\varepsilon}
\newcommand{\acknowledgments}{Acknowledgment}

\begin{document}

\title[Validity of formal expansions]
{Validity of formal expansions for singularly perturbed competition-diffusion systems}

\author{Ryunosuke Mori}
\address{Graduate School of Mathematical Sciences, 
University of Tokyo, Tokyo 153-8914, Japan}
%\curraddr{Department of Mathematics and Statistics,
%Case Western Reserve University, Cleveland, Ohio 43403}
\email{moriryu@ms.u-tokyo.ac.jp}

%\author{Author Two}
%\address{Mathematical Research Section, School of Mathematical Sciences,
%Australian National University, Canberra ACT 2601, Australia}

%    General info
\subjclass[2010]{Primary 35K57, 35B25; Secondary 35K40, 35B53}

\date{}


\keywords{Singular perturbation, asymptotic expansion, front profile, 
competition-diffusion, bistable, Liouville type theorem}

\begin{abstract}
We consider a two-species competition-diffusion system involving 
a small parameter $\ep>0$ and discuss the validity of formal asymptotic expansions of 
solutions near the sharp interface limit $\ep\approx0$. We assume that the 
corresponding ODE system has two stable equilibria. As in the scalar Allen--Cahn equation, 
it is known that the motion of the sharp interfaces of such systems is 
governed by the mean curvature flow with a driving force.  
The formal expansion also suggests that the profile of the transition layers 
converges to that of a traveling wave solution as $\ep\rightarrow0$. 
In this paper, we rigorously verify this latter ansatz 
for a large class of initial data. 

The proof relies on a rescaling argument, the super--subsolution method 
and a Liouville type theorem for eternal solutions of parabolic systems. 
Roughly speaking, 
the Liouville type theorem states that any eternal solution that lies between two traveling waves 
is itself a traveling wave. The same Liouville type theorem was established for the scalar 
Allen--Cahn equation by Berestycki and Hamel. 
In view of their importance, we prove the Liouville type theorems 
in a rather general framework, not only for 
two-species competition-diffusion systems but also for $m$-species cooperation-diffusion 
systems possibly with time periodic or spatially periodic coefficients. 
\end{abstract}

\maketitle

%\section*{This is an unnumbered first-level section head}
%This is an example of an unnumbered first-level heading.

%% The correct journal style for \specialsection is all uppercase; a known bug
%% in amsart.cls prevents this, so input must be uppercase until it is fixed.
%\specialsection*{This is a Special Section Head}
%\specialsection*{THIS IS A SPECIAL SECTION HEAD}
%This is an example of a special section head%
%%%%%%%%%%%%%%%%%%%%%%%%%%%%%%%%%%%%%%%%%%%%%%%%%%%%%%%%%%%%%%%%%%%%%%%%
%\footnote{Here is an example of a footnote. Notice that this footnote
%text is running on so that it can stand as an example of how a footnote
%with separate paragraphs should be written.
%\par
%And here is the beginning of the second paragraph.}%
%%%%%%%%%%%%%%%%%%%%%%%%%%%%%%%%%%%%%%%%%%%%%%%%%%%%%%%%%%%%%%%%%%%%%%%%
%%%%%%%%%%%%%%%%%%%%%%%%%%%%%%%%%%%%%%%%%
\section{Introduction} 
%%%%%%%%%%%%%%%%%%%%%%%%%%%%%%%%%%%%%%%%%
We consider the following Lotka--Volterra competition-diffusion system:
\begin{equation}\label{eq:crds}
\left\{
\begin{array}{ll}\vspace{6pt}
\ep u_{t}=\ep D_1\nabla\cdot(k(x)\nabla u) 
+\frac{h(x)}{\ep}(R_1-a_1u-b_1 v)u,&x\in\Omega,\ t>0,\\
\vspace{6pt}
\ep v_{t}=\ep D_2 \nabla\cdot(k(x)\nabla v) 
+ \frac{h(x)}{\ep}(R_2-a_2 u-b_2v)v,&x\in\Omega,\ t>0,\\
\vspace{6pt}
\partial u/\partial\nu=\partial v/\partial \nu=0,&x\in\partial\Omega,\ t>0,\\
u(x,0)=u_0(x),\ v(x,0)=v_0(x),&x\in\Omega,
\end{array}
\right.
\end{equation}
where 
$\ep$ is a positive parameter, 
$\Omega$ is a bounded domain in $\R^N$, 
$\partial /\partial\nu$ is the outward normal derivative on $\partial\Omega$,  
$R_i, a_i, b_i$, $D_i$ $(i=1,2)$ are positive constants and 
$k(x)$, $h(x)$ are positive smooth functions. Our focus is on the behavior of solutions when $\ep$ 
is very small. 

In the case of scalar Allen--Cahn equation, its singular limit has been studied 
by many researchers. 
It is known that, when $\ep$ is very small, solutions starting from  rather general 
initial data develop steep transition layers --- or interface --- within a very short time 
(generation of interface), and that the motion of these transition layers is well 
approximated by the spatially heterogeneous mean curvature flow (motion of interface). 
There is extensive literature on this subject, particularly on the motion of interface. 
We do not give a large list of references here. On the other hand, there are much fewer 
rigorous studies that cover both the generation and the motion of interface; see for example, 
\cite{AHM, AM, C, MS1}. 
In many of those studies, 
formal asymptotic expansions near the transition layers are used 
to make a rough approximation of the actual behavior of solutions and are also used to 
construct super- and subsolutions to establish the limit motion law of the sharp interface rigorously. 
X. Chen \cite{C} shows that the Hausdorff distance between 
the layer of the actual solution and the limit interface is of order $O(\ep|\log\ep|)$ 
for rather general initial data. Alfaro, Matano and Hilhorst \cite{AHM} improve this 
interface error estimate to $O(\ep)$. 

As regards the profile of interface, 
Bellettini and Paolini \cite{BP} and de Mottoni and Schatzman \cite{MS2} 
show that the real solution is well approximated 
by the formal expansion within an error margin of $O(\ep^2|\log|^2)$ and $O(\ep^2)$, 
at least on a finite time interval, provided that the initial data is already sufficiently close to 
the formal expansion. 
However, whether the actual solutions that start from arbitrary initial data really 
possess a profile predicted by the formal expansion or not remained long open. 
In Alfaro and Matano \cite{AM}, 
this question was answered affirmatively for a large class of initial data 
by showing rigorously that the solution converges uniformly to 
the principal term of the formal expansion as $\ep\rightarrow0$. 

In the case of the two-species competition-diffusion system of the form \eqref{eq:crds}, 
its singular limit has been studied by Hilhorst et al. \cite{HKMN}. 
They prove that the width of the transition layer is of order $O(\ep)$ and that 
the interface converges  as $\ep\rightarrow0$ 
to a time-dependent hypersurface whose motion 
is governed by the mean curvature flow with a driving force. 
However, to what extent the formal expansion represents the actual profile of the solution 
was not studied. 
Our goal is to prove the validity of this 
formal expansion; namely, we show that the solution profile of \eqref{eq:crds} 
near the interface converges uniformly to the principal term of the formal expansion for 
a rather general class of initial data. 

Throughout this paper, we assume 
\begin{equation}\label{cd:sc}
\frac{a_1}{a_2}<\frac{R_1}{R_2}<\frac{b_1}{b_2}.
\end{equation}
under this assumption, the corresponding ODE system 
\begin{equation}\label{eq:cods}
\left\{
\begin{array}{ll}\vspace{6pt}
\dot{u}=f(u,v),&t\in\R,\\
\vspace{6pt}
\dot{v}=g(u,v),&t\in\R,\\
u(0;u_0,v_0)=u_0,&v(0;u_0,v_0)=v_0
\end{array}
\right.
\end{equation}
has precisely four equilibria: two stable nodes
\[ 
p^+:=(R_1,0),\ p^-:=(0,R_2), 
\]
a saddle point 
\[
(u^*,v^*):=\Big(\frac{b_2 R_1-b_1 R_2}{a_1 b_2-a_2 b_1},
\frac{a_1 R_2-a_2 R_1}{a_1 b_2-a_2 b_1}\Big)
\]
and an unstable node $(0,0)$. 
Here $\dot{u}=\frac{d u}{d t}$ and 
\[
f(u,v):=(R_1-a_1u-b_1 v)u,\ \ g(u,v):=(R_2-a_2 u-b_2v)v.
\] 
We also assume: 

\vskip 6pt
\noindent
{\bf Assumption 1.} The following system has a solution:
\begin{equation}\label{eq:tw}
\left\{
\begin{array}{ll}\vspace{6pt}
D_1 U''  + f(U,V)=0,&z\in\R,\\
\vspace{6pt}
D_2 V'' + g(U,V)=0,&z\in\R,\\
\vspace{6pt}
(U(-\infty),V(-\infty))=(R_1,0),\\ 
(U(+\infty),V(+\infty))=(0,R_2).
\end{array}
\right.
\end{equation} 

\vskip 6pt

This assumption implies that the diffusion system
\begin{equation}\label{eq:1 dim cds}
\left\{
\begin{array}{ll}\vspace{6pt}
U_t=D_1 U_{zz} + f(U,V),&z\in\R,\ t\in\R,\\
V_t=D_2 V_{zz} + g(U,V),&z\in\R,\ t\in\R
\end{array}
\right.
\end{equation}
has a stationary wave solution. 

The existence and uniqueness of 
a traveling wave solution of \eqref{eq:1 dim cds} are shown by Kan-on \cite{K} 
under the condition \eqref{cd:sc}. 
The paper also shows continuously dependence of the traveling wave speed 
on the coefficients of the competition-diffusion system. 

As $\ep \rightarrow 0$, by a formal asymptotic analysis, 
the solution $(u^{\ep},v^{\ep})$ of \eqref{eq:crds} tends to a step function 
whose values are $(R_1,0)$, $(0,R_2)$ 
and the boundary $\Gamma(t)$ of the domain in which $(u^{\ep},v^{\ep})$ converges to $(R_1,0)$ 
moves according to the following equation
\begin{equation}\label{eq:interface}
V=-(N-1)k(x)\kappa-\frac{\partial}{\partial n} k(x)
-\frac{2k(x)(C+1)}{K(x)}\frac{\partial}{\partial n}K(x).
\end{equation}
Here $V$ is the normal velocity, $\kappa$ is the mean curvature and 
$n$ is the unit normal vector of $\Gamma(t)$ and $C$ is a constant defined by 
$(2.22)$ in \cite{HKMN}. K(x) is defined by 
\[
K(x)=\sqrt{\frac{h(x)}{k(x)}}
\]
(for more details, see Section 2 of \cite{HKMN}). 

Let $S$ denote the stable manifold of $(u^*,v^*)$ of \eqref{eq:cods}, that is, 
\[
S:=
\{(\xi,\eta)\in\R_+\times\R_+\mid
\lim_{\tau\rightarrow\infty}(u(\tau;\xi,\eta),v(\tau;\xi,\eta))
=(u^*,v^*)\},
\]
where $\R_+:=(0,\infty)$ and $(u(\tau;\xi,\eta),v(\tau;\xi,\eta))$ 
is a solution of \eqref{eq:cods} with initial data $(\xi,\eta)$. 
$S$ is called a separatrix and
\[
(\R_+\times\R_+)\backslash S=\Delta_1\cup\Delta_2,
\]
where 
\[
\begin{split}
\Delta_1:=
\{(\xi,\eta)\in\R_+\times\R_+\mid
\lim_{\tau\rightarrow\infty}(u(\tau;\xi,\eta),v(\tau;\xi,\eta))
=(R_1,0)\},\\
\Delta_2:=
\{(\xi,\eta)\in\R_+\times\R_+\mid
\lim_{\tau\rightarrow\infty}(u(\tau;\xi,\eta),v(\tau;\xi,\eta))
=(0,R_2)\}.
\end{split}
\]
For the proof of this result, see 
Chapter 12 of Hirsch and Smale \cite{HS}. 

\begin{remark}\label{rm:S and tw}
The stable manifold $S$ of $(u^*,v^*)$ of \eqref{eq:cods} can be described as follows.
\begin{equation}\label{eq:separatrix}
S=\{(u,v)\in\R_+\times\R_+\mid H(u,v)=0\},
\end{equation}
where $H\in C(\overline{\R_+}\times\overline{\R_+})\cap 
C^{\infty}(\R_+\times\R_+)$ satisfies 
\[
H(0,0)=0,\ H_u(u,v)<0\ \ 
{\rm and}\ \ H_v(u,v)>0\ \ {\rm for }\ \ u>0,\ v>0. 
\]
Moreover there is a function $\zeta\in C(\overline{\R_+})\cap C^{\infty}(\R_+)$ such that 
$\zeta'(u)>0$ for $u>0$ and 
\[
\begin{split}
&\zeta(u^*)=v^*,\ \ H(u,v)=v-\zeta(u)\  \ {\rm for}\ \ u\geq0,\ v\geq0\ \ 
{\rm or}\\ 
&\zeta(v^*)=u^*,\ \ 
H(u,v)=\zeta(v)-u\ \ {\rm for}\ \ u\geq0,\ v\geq0.
\end{split}
\]
In fact, by geometric theory of ODE systems, 
\eqref{eq:cods} has a locally stable manifold 
\[
\{(u,\zeta(u))\mid u^*-\ep_0<u<u^*+\ep_0\},
\]
where $\zeta$ is a smooth function satisfying 
$\zeta'(u^*)>0$, $\zeta(u^*)=v^*$ and 
\begin{equation}\label{eq:stable manifold}
\frac{d \zeta}{du}(u)=\frac{g(u,\zeta(u))}{f(u,\zeta(u))}.
\end{equation}
Let $\zeta(u)$, $u\in(U_*,U^*)$ be a function satisfying 
$\zeta'(u^*)>0$, $\zeta(u^*)=v^*$ and \eqref{eq:stable manifold} 
which has the maximal interval of existence. 
From $\zeta'(u^*)>0$, $\zeta(u^*)=v^*$, \eqref{eq:stable manifold} and 
\[
\begin{split}
&f>0,\ g>0\ \ {\rm on}\ \ (0,u^*)\times(0,v^*),\\
&f<0,\ g<0\ \ {\rm on}\ \ (u^*,\infty)\times(v^*,\infty),
\end{split}
\]
$\zeta$ is strictly increasing in $(U_*,U^*)$. 
Since, for each $(u_0,v_0)\in(0,u^*)\times(0,v^*)$, 
the solution $(u(t;u_0,v_0),v(t;u_0,v_0))$ of \eqref{eq:cods} tends to $(0,0)$ 
as $t\rightarrow-\infty$, 
\[
U_*=0\ \ {\rm and}\ \ \underset{u\rightarrow U_*}{\lim} \zeta(u)=0.
\] 
Furthermore it is also easily obtained that 
\begin{itemize}
\item[(1)]$U^*=\infty$ or
\item[(2)]$U^*<\infty$ and  
$\underset{u\rightarrow U^*}{\lim}\zeta(u)=\infty$. 
\end{itemize}
In the case that ($1$) holds, we may put $H(u,v)=v-\zeta(u)$. 
In the case that ($2$) holds, we may put $H(u,v)=\zeta^{-1}(v)-u$. 
Therefore \eqref{eq:separatrix} holds. 

On the other hand, a solution $(U,V)$ of \eqref{eq:tw} 
satisfies 
\[
U'(z)<0,\ V'(z)>0\ \ (z\in\R).
\] 
Therefore if $H(u,v)=v-\zeta(u)$, then 
\[
\begin{split}
&H(U(-\infty),V(-\infty))=H(R_1,0)=-\zeta(R_1)<0,\\
&H(U(+\infty),V(+\infty))=H(0,R_2)=R_2>0,
\end{split}
\]
\begin{equation}\label{eq:non degenerateness}
\frac{d}{d z}H(U,V)=(H_u(U,V),H_v(U,V))\cdot(U',V')>0.
\end{equation}
Hence $(U(z),V(z))$ $(z\in\R)$ and $S=\{(u,v)\mid H(u,v)=0\}$ intersect 
at exact one point, transversely. 
\end{remark}

We define $\Gamma_0$ as follows:
\[
\Gamma_0:=\{x\in\Omega\mid (u_0(x),v_0(x))\in S\}
\]
and we assume: 

\vskip 6pt 
\noindent
{\bf Assumption 2.} $u_0$, $v_0$ are continuous on $\overline{\Omega}$ 
and satisfy $|u_0|+|v_0|>0$ on $\overline{\Omega}$.

\vskip 6pt 
\noindent
{\bf Assumption 3.} $\Gamma_0$ is a smooth closed hypersurface in $\Omega$ 
and satisfies $\Gamma_0\cap\partial\Omega=\emptyset$.

\vskip 6pt 
\noindent
{\bf Assumption 4.} The classical solution $\Gamma(t)$ of \eqref{eq:interface} 
with initial data $\Gamma(0)=\Gamma_0$ exists on an interval $0\leq t\leq T$ 
and is a smooth closed hypersurface in $\Omega$ for every $t\in[0,T]$.

\vskip 6pt 
\noindent
{\bf Assumption 5.} There exists a constant $A_0>0$ such that 
\[
{\rm dist}_{\R^2}((u_0(x),v_0(x)),S)\geq A_0{\rm dist}(x,\Gamma_0),\ \ x\in\Omega,
\] 
where 
\[{\rm dist}_{\R^2}((u,v),S)
:=\underset{(\xi,\eta)\in S}{\inf}|(u-\xi,v-\eta)|,\ \ 
{\rm dist}(x,\Gamma_0)
:=\underset{y\in \Gamma_0}{\inf}|x-y|.
\]

\vskip 6pt 

\begin{remark}\label{rm:maximal time}
Under Assumption 3, by coordinate transformation and 
a theorem for a quasilinear parabolic equation in Lunardi \cite{L}, there exists $T_{\max}>0$ such that 
\eqref{eq:interface} possesses a unique smooth solution 
$\Gamma(t)$, $0\leq t<T_{\max}$. In the sequel, we can select any 
$T\in(0,T_{\max})$ in Assumption 4. 
\end{remark}

The hypersurface $\Gamma(t)$ divides $\Omega$ into two connected components, 
the inside of $\Gamma(t)$ and the outside of $\Gamma(t)$, denoted by 
$\Omega_{in}(t)$ and $\Omega_{out}(t)$, respectively. 
As in \cite{HKMN}, we may assume that 
$(u_0(x),v_0(x))$ satisfies 
\[
\Omega_{in}(0)=\{x\mid (u_0(x),v_0(x))\in\Delta_1\},\ \ 
\Omega_{out}(0)=\{x\mid (u_0(x),v_0(x))\in\Delta_2\}.
\]

Let $(u^{\ep},v^{\ep})$ be the solution of \eqref{eq:crds} 
and define $\Gamma^{\ep}(t)$, $\Omega_{in}^{\ep}(t)$, 
$\Omega_{out}^{\ep}(t)$ as follows:
\[
\begin{split}
&\Gamma^{\ep}(t):=\{x\mid (u^{\ep}(x,t),v^{\ep}(x,t))\in S\},\\
&\Omega_{in}^{\ep}(t):=\{x\mid (u^{\ep}(x,t),v^{\ep}(x,t))\in\Delta_1\},\\
&\Omega_{out}^{\ep}(t):=\{x\mid (u^{\ep}(x,t),v^{\ep}(x,t))\in\Delta_2\}.
\end{split}
\]

Let $d(x,t)$, $d^{\ep}(x,t)$ be the signed distance functions associated with 
$\Gamma(t)$, $\Gamma^{\ep}(t)$, respectively, that is, 
\[
d(x,t):=
\left\{
\begin{array}{l}\vspace{6pt}
-{\rm dist}(x,\Gamma(t))\ {\rm if}\ x\in\Omega_{in}(t),\\
{\rm dist}(x,\Gamma(t))\ {\rm if}\ x\in\Omega_{out}(t),
\end{array}
\right.
\]
\[
d^{\ep}(x,t):=
\left\{
\begin{array}{l}\vspace{6pt}
-{\rm dist}(x,\Gamma^{\ep}(t))\ {\rm if}\ x\in\Omega_{in}^{\ep}(t),\\
{\rm dist}(x,\Gamma^{\ep}(t))\ {\rm if}\ x\in\Omega_{out}^{\ep}(t).
\end{array}
\right.
\]

Now we state our main theorem.

\begin{theorem}\label{th:validity}
Let Assumptions 1,2,3,4 and 5 hold. Let $(u^{\ep},v^{\ep})$ be 
the solution of \eqref{eq:crds} and 
let $(\phi_0,\psi_0)$ be the solution of \eqref{eq:tw} satisfying 
$(\phi_0(0),\psi_0(0))\in S$. 
Put 
\[
t^{\ep}:=\ep^2|\log \ep|\ \ {\rm and}\  \ 
(U_0(\zeta,x),V_0(\zeta,x)):=(\phi_0(K(x)\zeta),\psi_0(K(x)\zeta)).
\]
Then there exists a constant $C>0$ such that the following hold for arbitrary $\mu>1$.
\begin{enumerate}
\item[(i)] If $\ep$ is small enough, then, for each $t\in[\mu Ct^{\ep},T]$, 
$\Gamma^{\ep}(t)$ can be expressed as a graph of a smooth function 
$\eta^{\ep}(\cdot,t)$ over $\Gamma(t)$ whose norm 
$\|\eta^{\epsilon}(\cdot,t)\|_{L^{\infty}(\Gamma(t))}$ and 
gradient $\nabla_{\Gamma(t)}\eta(x,t)$ 
on $\Gamma(t)$ tend to $0$ as $\ep\rightarrow0$ 
uniformly for $x\in\Gamma(t)$ and $t\in[\mu Ct^{\ep},T]$.
\item[(ii)] Let $d^{\ep}$ be the signed distance function associated with $\Gamma^{\ep}$. 
Then 
\begin{equation}\label{eq:validity 1}
\left\{
\begin{split}
\lim_{\ep\rightarrow0}\sup_{\mu Ct^{\ep}\leq t\leq T,\,x\in\overline{\Omega}}
\Big|u^{\ep}(x,t)-
U_0\Big(\frac{d^{\ep}(x,t)}{\ep},x\Big)
\Big|
=0,\\
\lim_{\ep\rightarrow0}\sup_{\mu Ct^{\ep}\leq t\leq T,\,x\in\overline{\Omega}}
\Big|v^{\ep}(x,t)-
V_0\Big(\frac{d^{\ep}(x,t)}{\ep},x\Big)
\Big|
=0.
\end{split}
\right.
\end{equation}
\item[(iii)] There exists a family of functions 
\[
\theta^{\ep}:\underset{0\leq t\leq T}{\cup}(\Gamma(t)\times\{t\})\rightarrow\R
\]
whose $L^{\infty}$-norms are bounded as $\ep\rightarrow0$, such that 
\begin{equation}\label{eq:validity 2}
\left\{
\begin{split}
\lim_{\ep\rightarrow0}\sup_{\mu Ct^{\ep}\leq t\leq T,\,x\in\overline{\Omega}}
\Big|u^{\ep}(x,t)-
U_0\Big(\frac{d(x,t)-\ep\theta^{\ep}(p(x,t),t)}{\ep},x\Big)
\Big|
=0,\\
\lim_{\ep\rightarrow0}\sup_{\mu Ct^{\ep}\leq t\leq T,\,x\in\overline{\Omega}}
\Big|v^{\ep}(x,t)-
V_0\Big(\frac{d(x,t)-\ep\theta^{\ep}(p(x,t),t)}{\ep},x\Big)
\Big|
=0,
\end{split}
\right.
\end{equation}
where $d$ denotes the signed distance function associated with $\Gamma$ 
and $p(x,t)$ denotes a point on $\Gamma(t)$ satisfying 
${\rm dist}(x,\Gamma(t))=|x-p(x,t)|$.
\end{enumerate}
\end{theorem}
The statement (i) means that the interface $\Gamma^{\epsilon}(t)$ of the solution converges 
to the hypersurface $\Gamma(t)$ as $\epsilon\rightarrow0$ in the $C^1$ topology, 
where $\Gamma(t)$ is the classical solution of \eqref{eq:interface} given in Assumption 4. 
The statement (iii) implies that the principal term of the formal expansion gives uniform approximation of the real solution.  
In the proof of Theorem \ref{th:validity}, we use an idea similar to what is found in \cite{AM}
for the scalar Allen--Cahn equation. Namely, the proof is based on 
a rescaling argument, the super--subsolution method and 
a Liouville type result for eternal solutions of some competition-diffusion systems. 
However, in our problem, the interface $\Gamma^{\epsilon}(t)$ of the solution 
is defined as the inverse image of the one-dimensional separatrix $S$ rather than that of a single point, 
which makes the estimates of the distance between $\Gamma^{\epsilon}(t)$ and $\Gamma(t)$ 
more involved than in the scalar Allen--Cahn case.  
We also need to extend the Liouville type results in Berestycki and Hamel \cite{BH1,BH2} 
to parabolic systems. 

The rest of the paper is organized as follows. In Section \ref{sc:Ltype}, 
we state the Liouville type theorems for eternal solutions of parabolic systems 
which play a key rule in proving the main theorem. 
Though what we need in the proof of Theorem \ref{th:validity} is the Liouville type 
theorem for a two-species conpetition-diffusion equation, in view of the importance 
of such Liouville type theorems, we present the results in a more general setting, 
namely for $m$-species cooperation-diffusion systems possibly with 
spatially periodic or time periodic coefficients. 
In Section \ref{sc:proof of main}, 
we state two lemmas, Lemmas \ref{lm:sandwich} and \ref{lm:Ltype} 
that are used in the proof of Theorem \ref{th:validity} and 
we prove Theorem \ref{th:validity}. At the end of Section \ref{sc:proof of main}, 
we prove 
Lemmas \ref{lm:sandwich} and \ref{lm:Ltype}. 
In Section \ref{sc:proof of Ltype}, we prove the Liouville type theorems.
%%%%%%%%%%%%%%%%%%%%%%%%%%%%%%%%%%%%%%%%%
%%%%%%%%%%%%%%%%%%%%%%%%%%%%%%%%%%%%%%%%%
\section{Liouville type theorems for eternal solutions of a parabolic system}
\label{sc:Ltype}
Before starting the proof of the main theorem, 
we present Liouville type theorems for eternal solutions of reaction-diffusion systems. 
These are extensions of similar 
Liouville type results of Berestycki and Hamel \cite{BH1,BH2} to systems of reaction-diffusion 
equations. 
Though what we need in the proof of the main theorem is only a special case of such 
Liouville type results, 
we state them in a rather general setting since we think those results are important 
in their own right.
%%%%%%%%%%%%%%%%%%%%%%%%%%%%%%%%%%%%%%%%%
\subsection{Statement of Liouville type theorems (homogeneous case)}
Let us first state a result on a reaction-diffusion system of the form:
\begin{equation}\label{eq:crds 2}
\left\{
\begin{split}
&u_t=D_1\Delta u+ f_1(u,v),\ \ x\in\R^N,\ t\in\R,\\
&v_t=D_2\Delta v +f_2(u,v),\ \ x\in\R^N,\ t\in\R,
\end{split}
\right.
\end{equation}
where 
$D_1$, $D_2$ are positive constants and 
$f_1, f_2$ are smooth functions such that \eqref{eq:crds 2} 
is a {\it competition-diffusion system}, that is, it holds that 
\begin{equation}\label{cd:irreducible}
f_{1,v}=\partial_{v}f_1<0,\ f_{2,u}<0\ \ {\rm in}\ \ 
(p_1^{-},p_1^{+})\times(p_2^{+},p_2^{-}).
\end{equation}
Furthermore, we assume that 
$F=(f_1,f_2)$
has two linearly stable equilibria 
\[
p^+=(p^+_1,p^+_2),\ 
p^-=(p^-_1,p^-_2)\ \  ( p^-_1<p^+_1,\  p^+_2<p^-_2),
\] 
that is, for some constants $\lambda_{\pm}>0$ and vectors 
$\varphi^{\pm}={}^t (\varphi^{\pm}_1,\varphi^{\pm}_2)$ 
$(\varphi_2^{\pm}<0<\varphi^{\pm}_1)$,
\begin{equation}\label{cd:bistable}
F(p^{\pm})=(0,0),\ \ D F(p^{\pm})\varphi^{\pm}=-\lambda_{\pm}\varphi^{\pm},
\end{equation}
where 
\[
D F(p^{\pm})=
\left(
\begin{array}{cc}
f_{1,u}(p^{\pm})&f_{1,v}(p^{\pm})\\
f_{2,u}(p^{\pm})&f_{2,v}(p^{\pm})
\end{array}
\right).
\] 
We also assume: 
\begin{equation}\tag{A}
\left\{
\begin{split}
&\eqref{eq:crds 2}\ {\rm has\ a\ traveling\ wave\ solution}\\ 
&(u(x,t),v(x,t))=(\phi_1(n\cdot x-ct),\phi_2(n\cdot x-ct))\\ 
&{\rm with\ a\ direction}\ n\in\R^N\ 
(|n|=1)\ \ {\rm and\ a\ speed}\ c\in\R\  {\rm satisfying}\\ 
&(\phi_1(\mp\infty),\phi_2(\mp\infty))=p^{\pm}\ \ 
{\rm and}\ 
\phi_1'<0,\ \phi_2'>0\ {\rm in}\ \R.
\end{split}
\right.
\end{equation}
\begin{remark}
The assumption $({\rm A})$ means that the system \eqref{eq:crds 2} has a planar 
wave solution whose direction and speed are $n$ and $c$, respectively. 
\end{remark}

\begin{theorem}
[Liouville type theorem for a competition-diffusion system]
\label{th:Ltype}
Assume $({\rm A})$, \eqref{cd:bistable} 
and \eqref{cd:irreducible}. 
Let $(u(x,t),v(x,t))$ $(x\in\R^N,\ t\in\R)$ be a solution of \eqref{eq:crds 2} which satisfies that 
there are a unit vector $n$, some constants $c\in\R$, $a<b$ such that, 
for all $(x,t)\in\R^N\times\R$, 
\begin{equation}\label{eq:sandwich 1}
\left\{
\begin{split}
&\phi_1(n\cdot x-ct-a)\leq u(x,t)\leq\phi_1(n\cdot x-ct-b),\\
&\phi_2(n\cdot x-ct-b)\leq v(x,t)\leq \phi_2(n\cdot x-ct-a),
\end{split}
\right.
\end{equation}
where $(\phi_1,\phi_2)$ is a function satisfying $({\rm A})$ with the speed $c$. 
Then there exists $\theta_0\in(a,b)$ such that, for all $(x,t)\in\R^N\times\R$,
\[
(u(x,t),v(x,t))=(\phi_1(n\cdot x-ct-\theta_0),\phi_2(n\cdot x-ct-\theta_0)).
\]
\end{theorem}
Roughly speaking, Theorem \ref{th:Ltype} means 
that eternal solutions that are sandwiched between two planar wave solutions 
are precisely planar waves. 
The following Theorems \ref{th:Ltype 1}, \ref{th:Ltype 2} and \ref{th:Ltype 3} have similar meaning. 

Let us reformulate Theorem \ref{th:Ltype} in more general settings. 
As one easily sees, 
a two-species competition system can be converted to a two-species cooperation-diffusion system 
by the change of variables $(u,v)\mapsto(u,-v)$. Therefore it suffices  to state the results 
for $m$-species cooperation-diffusion systems. 

First, let us define  
order relations in $\R^k$ and in $X=C(\R^l:\R^k)$ as follows:  
\[
\begin{split}
&(u_1,u_2,\cdots,u_k)\preceq (v_1,v_2,\cdots,u_k)\ \ {\rm if}\ \ u_i\leq v_i,\ \ (i=1,2,\cdots,k),\\
&u:=(u_1,u_2,\cdots,u_k)\prec v:=(v_1,v_2,\cdots,v_k)\ \ 
{\rm if}\ \ u\preceq v 
\ \ {\rm and}\ \ u\not= v,\\
&(u_1,u_2,\cdots,u_k)\ll (v_1,v_2,\cdots,v_k)\ \ {\rm if}\ \ u_i< v_i,\ \ (i=1,2,\cdots,k),
\end{split}
\]
\[
\begin{split}
&u\preceq v\ \ {\rm if}\ \ u(x)\preceq v(x)\ \ {\rm for\ all}\ \ x\in\R^l,\\
&u\prec v\ \ {\rm if}\ \ u\preceq v\ \ {\rm and}\ \ u\not=v,\\
&u\ll v\ \ {\rm if}\ \ u(x)\ll v(x)\ \ {\rm for\ all}\ \ x\in\R^l.
\end{split}
\]
Now we consider a reaction-diffusion system of the form:
\begin{equation}\label{eq:crds 3}
\left\{
\begin{split}
u_{1,t}=&D_1\Delta u_1+f_1(u_1,u_2,\cdots,u_m),\ \ x\in\R^N,\ t\in\R,\\
u_{2,t}=&D_2\Delta u_2+f_2(u_1,u_2,\cdots,u_m),\ \ x\in\R^N,\ t\in\R,\\
&\vdots\\
u_{m,t}=&D_m\Delta u_m+f_m(u_1,u_2,\cdots,u_m),\ \ x\in\R^N,\ t\in\R,
\end{split}
\right.
\end{equation}
where  
$D_l$ $(l=1,2,\cdots,m)$ are positive constants and 
$f_1,f_2,\cdots,f_m$ are smooth functions such that  
\eqref{eq:crds 3} is a {\it cooperation-diffusion system}, that is, 
it holds that 
\begin{equation}\label{cd:cs 1}
f_{k,u_l}=\partial_{u_l}f_k\geq0\ \ (k\not=l)\ \ 
{\rm in}\ \ 
(p^-,p^+):=
(p_1^{-},p_1^{+})\times(p_2^{-},p_2^{+})\times\cdots\times(p_m^{-},p_m^{+}),
\end{equation} 
\begin{equation}\label{cd:irreducible 1}
DF(u)\ {\rm is\ an\ irreducible\ matrix}
\ {\rm for\ each}\ 
u\in(p^{-},p^{+}),
\end{equation}
where 
\[
D F(p^{\pm})=
\left(
\begin{array}{cccc}
f_{1,u_1}(p^{\pm})&f_{1,u_2}(p^{\pm})&\cdots&f_{1,u_m}(p^{\pm})\\
f_{2,u_1}(p^{\pm})&f_{2,u_2}(p^{\pm})&\cdots&f_{2,u_m}(p^{\pm})\\
\vdots&\vdots&\ddots&\vdots\\
f_{m,u_1}(p^{\pm})&f_{m,u_2}(p^{\pm})&\cdots&f_{m,u_m}(p^{\pm})
\end{array}
\right).
\] 
We say that an $m\times m$ matrix $A=(a^{kl})$ is {\it reducible} if 
there is $\emptyset\not=\Lambda\subsetneq\{1,2,\cdots,m\}$ such that 
\[
a^{kl}=0\ \ {\rm for}\ \ k\in\Lambda,\ l\not\in\Lambda.
\]
We say that an $m\times m$ matrix $A$ is {\it irreducible} if $A$ is not reducible.

Furthermore, we assume that 
$F=(f_1,f_2,\cdots,f_m)$
has two linearly stable equilibria 
\[
p^+=(p^+_1,p^+_2,\cdots,p^+_m)\gg
p^-=(p^-_1,p^-_2,\cdots,p^-_m),
\] 
that is, for some constants $\lambda_{\pm}>0$ and unit vectors
\[
\varphi^{\pm}={}^t (\varphi^{\pm}_1,\varphi^{\pm}_2,\cdots,\varphi^{\pm}_m)
\gg {}^t (0,0,\cdots,0),
\]
\begin{equation}\label{cd:bistable 1}
F(p^{\pm})=(0,0,\cdots,0),\ \ D F(p^{\pm})\varphi^{\pm}=-\lambda_{\pm}\varphi^{\pm}.
\end{equation}
We also assume: 
\begin{equation}\tag{A'}
\left\{
\begin{split}
&\eqref{eq:crds 3}\ {\rm has\ a\ traveling\ wave\ solution}\ 
u(x,t)=\phi(n\cdot x-ct)\\ 
&{\rm with\ a\ direction}\ n\in\R^N\ (|n|=1)\ \ {\rm and\ a\ speed}\ c\in\R\\  
&{\rm satisfying}\ \phi(\mp\infty)
=p^{\pm}\ {\rm and}\ 
\phi'\ll(0,0,\cdots,0)\ {\rm in}\ \R.
\end{split}
\right.
\end{equation}

\begin{theorem}
[Liouville type theorem for a cooperation-diffusion system]
\label{th:Ltype 1}
Assume $({\rm A'})$, \eqref{cd:bistable 1} and \eqref{cd:cs 1}. 
Let 
$
u(x,t)\ \ (x\in\R^N,\ t\in\R)
$ 
be a solution of \eqref{eq:crds 3} which satisfies that 
there are a unit vector $n$, some constants $c\in\R$, $a<b$ such that, 
for all $(x,t)\in\R^N\times\R$, 
\begin{equation}\label{eq:sandwich 1}
\phi(n\cdot x-ct-a)\preceq u(x,t)\preceq\phi(n\cdot x-ct-b),
\end{equation}
where $\phi$ is a function satisfying $({\rm A'})$ 
with the speed $c$. 
Then there exists a function $\widetilde{\phi}$ satisfying $({\rm A'})$
such that $u(x,t)=\widetilde{\phi}(n\cdot x-ct)$.
If, in addition, assume \eqref{cd:irreducible 1}, then  
there exists $\theta_0\in(a,b)$ such that 
\[
u(x,t)=\phi(n\cdot x-ct-\theta_0)\ \ 
{\rm for\ all}\ \ (x,t)\in\R^N\times\R.
\]
\end{theorem}

%%%%%%%%%%%%%%%%%%%%%%%%%%%%%%%%%%%%%%%%%
\subsection{Statement of Liouville type theorems (inhomogeneous case)}

The Liouville type theorem in the previous subsection can be extended to 
more general systems. 
First we consider the following time periodic systems: 
\begin{equation}\label{eq:t-crds}
\left\{
\begin{split}
u_{l,t}=&\sum_{i,j=1}^N D_l^{ij}(t)u_{l,x_ix_j}+q_l(t)\cdot\nabla u_l+f_l(t,u_1,\cdots,u_m)\\
&{\rm for }\ \ x\in\R^N,\ t\in\R\ \ (l=1,2,\cdots,m),
\end{split}
\right.
\end{equation}
where 
\[
\begin{split}
&D^{ij}_l(t)\ \ (i,j=1,2,\cdots,N),\ \ q_l(t)\in\R^N\ \ 
(l=1,2,\cdots,m),\\ 
&F(t,u_1,\cdots,u_m)=(f_1,f_2,\cdots,f_m)
\end{split}
\] 
are 
H$\ddot{\rm o}$lder continuous in $t$, smooth in $(u_1,u_2,\cdots,u_m)$ 
and $f_1,f_2,\cdots,f_m$ satisfy 
\begin{equation}\label{cd:t-cs}
f_{k,u_l}(t,u)=\partial_{u_l} f_k\geq0\ \ (k\not=l)\ \ 
{\rm for\ each}\ \ u\in(p^-(t),p^+(t)),\ t\in\R,
\end{equation} 
\begin{equation}\label{cd:t-irreducible}
DF(t,u)\ {\rm is\ an\ irreducible\ matrix}\ 
{\rm for\ each}\ u\in(p^-(t),p^+(t)),\ t\in\R,
\end{equation}
where 
\[
D F(t,p^{\pm})=
\left(
\begin{array}{cccc}
f_{1,u_1}(t,p^{\pm})&f_{1,u_2}(t,p^{\pm})&\cdots&f_{1,u_m}(t,p^{\pm})\\
f_{2,u_1}(t,p^{\pm})&f_{2,u_2}(t,p^{\pm})&\cdots&f_{2,u_m}(t,p^{\pm})\\
\vdots&\vdots&\ddots&\vdots\\
f_{m,u_1}(t,p^{\pm})&f_{m,u_2}(t,p^{\pm})&\cdots&f_{m,u_m}(t,p^{\pm}).
\end{array}
\right).
\]
We assume that there are $\alpha_2\geq\alpha_1>0$, $T>0$ such that 
\begin{equation}\label{cd:ellipticity}
\alpha_1|\xi|^2\leq\sum_{i,j=1}^N D^{ij}_l(t)\xi_i\xi_j\leq\alpha_2|\xi|^2\ \ 
{\rm for}\ \ t\in\R,\ \xi\in\R^N\ (l=1,2,\cdots,m),
\end{equation}
\begin{equation}\label{cd:T-periodic}
\begin{split}
&D^{ij}_l(t+T)=D^{ij}_l(t),\ q_l(t+T)=q_l(t),\ 
f_l(t+T,\cdot)=f_l(t,\cdot)\\ 
&{\rm for}\ \ t\in\R\ \ 
(i,j=1,2,\cdots,N;\ l=1,2,\cdots,m)
\end{split}
\end{equation}
and there are two smooth functions 
\[
p^+(t)=(p^+_1(t),p^+_2(t),\cdots,p^+_m(t))\gg 
p^-(t)=(p^-_1(t),p^-_2(t),\cdots,p^-_m(t))
\] 
such that, for some constants $\lambda_{\pm}>0$ and vector valued functions 
\[
\varphi^{\pm}={}^t (\varphi^{\pm}_1,\varphi^{\pm}_2,\cdots,\varphi^{\pm}_m)
\gg {}^t (0,0,\cdots,0),
\]  
\begin{equation}\label{cd:t-bistable}
\begin{split}
&\frac{d p^{\pm}}{dt}-F(t,p^{\pm})=(0,0,\dots,0),\ \ 
\frac{d \varphi^{\pm}}{d t}-D F(t,p^{\pm})\varphi^{\pm}=\lambda_{\pm}\varphi^{\pm},\\
&p^{\pm}(\cdot+T)=p^{\pm}(\cdot),\ 
\varphi^{\pm}(\cdot+T)=\varphi^{\pm}(\cdot).
\end{split}
\end{equation}
We also assume: 
\begin{equation}\tag{A1}
\left\{
\begin{split}
&\eqref{eq:t-crds}\ {\rm has\ a\ pulsating\ traveling\ wave\ solution}
\\ &
u(x,t)=\phi(n\cdot x-ct,t)\ 
{\rm with\ a\ direction}\ n\in\R^N\ {\rm and}\\ &
{\rm a\ speed}\ c\in\R\  {\rm satisfying}\ 
\phi(\mp\infty,t)
=p^{\pm}(t)\ {\rm and}\\
&\phi(z,t+T)=\phi(z,t),\ 
\phi_{z}(z,t)\ll(0,0,\cdots,0)
\ {\rm for}\ z\in\R,\ t\in\R.
\end{split}
\right.
\end{equation}
\begin{remark}
As shown in Bao and Wang \cite{BW}, 
the following time-periodic Lotka--Volterra competition-diffusion system 
\[
\left\{
\begin{array}{l}\vspace{6pt}
u_t=d_1(t)u_{xx}+u(r_1(t)-a_1(t)u-b_1(t)v),\\
v_t=d_2(t)v_{xx}+v(r_2(t)-a_2(t)u-b_2(t)v),
\end{array}
\right.
x\in\R,\ t\in\R
\]
satisfies $({\rm A1})$ under the assumption stated below 
(after the change of variables $(u,v)\mapsto(u,-v))$: 
\begin{itemize}
\item{} $d_i(t)$, $r_i(t)$, $a_i(t)$, $b_i(t)\in C^{\frac{\theta}{2}}(\R)$\ $(i=1,2; \theta\in(0,1))$ 
are $T$-periodic functions satisfying $d_i(t)>0$, $a_i(t)>0$, $b_i(t)>0$ for any $t\in[0,T]$, 
$\overline{r_i}=\frac{1}{T}\int_0^T r_i(t)d t>0$ and 
\[
\begin{split}
\overline{r_1}&<\underset{t}{\min}\Big(\frac{b_1(t)}{b_2(t)}\Big)\overline{r_2},\ \ 
\overline{r_2}<\underset{t}{\min}\Big(\frac{a_2(t)}{a_1(t)}\Big)\overline{r_1},\\ 
\overline{r_1}+\overline{r_2}&>
\underset{t}{\max}\Big(\frac{a_2(t)}{a_1(t)}\Big)\overline{r_1},\ \  
\overline{r_1}+\overline{r_2}>\underset{t}{\max}\Big(\frac{b_1(t)}{b_2(t)}\Big).
\end{split}
\]
\end{itemize}
\end{remark}
\begin{theorem}
[Liouville type theorem for $t$-periodic system]
\label{th:Ltype 2}
Assume $({\rm A1})$, \eqref{cd:ellipticity}, \eqref{cd:T-periodic}, \eqref{cd:t-bistable} 
and \eqref{cd:t-cs}. Let $u$ be a solution of \eqref{eq:t-crds} 
which satisfies that 
there are a unit vector $n$, some constants $c\in\R$, $a<b$ such that, 
for all $(x,t)\in\R^N\times\R$, 
\begin{equation}\label{eq:sandwich 1}
\phi(n\cdot x-ct-a,t)\preceq u(x,t)\preceq\phi(n\cdot x-ct-b,t),
\end{equation}
where $\phi$ is a function satisfying $(\rm T)$ with the speed $c$. 
Then there exists a function $\widetilde{\phi}$ satisfying $({\rm A1})$ 
such that $u(x,t)=\widetilde{\phi}(n\cdot x-ct,t)$.
If, in addition, assume \eqref{cd:t-irreducible}, then  
there exists $\theta_0\in(a,b)$ such that
\[
u(x,t)=\phi(n\cdot x-ct-\theta_0,t)\ \ 
{\rm for\ all}\ \ (x,t)\in\R^N\times\R.
\]
\end{theorem}

Next we state the theorem for spatially periodic systems: 
\begin{equation}\label{eq:x-crds}
\left\{
\begin{split}
u_{l,t}=&\sum_{i,j=1}^N D_l^{ij}(x)u_{l,x_ix_j}+q_l(x)\cdot\nabla u_l+f_l(x,u_1,\cdots,u_m)\\
&{\rm for }\ \ x\in\R^N,\ t\in\R\ \ (l=1,2,\cdots,m),
\end{split}
\right.
\end{equation}
where 
\[
\begin{split}
&D^{ij}_l(x)\ \ (i,j=1,2,\cdots,N),\ \ q_l(x)\in\R^N\ \ 
(l=1,2,\cdots,m), \\
&F(x,u_1,\cdots,u_m)=(f_1,f_2,\cdots,f_m)
\end{split}
\] 
are 
H$\ddot{\rm o}$lder continuous in $x$, smooth in $(u_1,u_2,\cdots,u_m)$ 
and $f_1,f_2,\cdots,f_m$ satisfy  
\begin{equation}\label{cd:x-cs}
f_{k,u_l}(x,u)=\partial_{u_l}f_k\geq0\ \ (k\not=l)\ \ 
{\rm for\ each}\ \ 
u\in(p^{-}(x),p^{+}(x)),\ 
x\in\R^N,
\end{equation} 
\begin{equation}\label{cd:x-irreducible}
DF(x,u)\ {\rm is\ an\ irreducible\ matrix}\ 
{\rm for\ each}\ u\in(p^-(x),p^+(x)),\ 
x\in\R^N,
\end{equation}
where 
\[
D F(x,p^{\pm})=
\left(
\begin{array}{cccc}
f_{1,u_1}(x,p^{\pm})&f_{1,u_2}(x,p^{\pm})&\cdots&f_{1,u_m}(x,p^{\pm})\\
f_{2,u_1}(x,p^{\pm})&f_{2,u_2}(x,p^{\pm})&\cdots&f_{2,u_m}(x,p^{\pm})\\
\vdots&\vdots&\ddots&\vdots\\
f_{m,u_1}(x,p^{\pm})&f_{m,u_2}(x,p^{\pm})&\cdots&f_{m,u_m}(x,p^{\pm})
\end{array}
\right).
\] 
We assume that there are $\alpha_2\geq\alpha_1>0$, $L_j>0$\ $(j=1,2,\cdots,N)$ such that 
\begin{equation}\label{cd:ellipticity 1}
\alpha_1|\xi|^2\leq\sum_{i,j=1}^N D^{ij}_l(x)\xi_i\xi_j\leq\alpha_2|\xi|^2\ \ 
{\rm for}\ \ x\in\R^N,\ \xi\in\R^N\ (l=1,2,\cdots,m),
\end{equation}
\begin{equation}\label{cd:L-periodic}
\begin{split}
&D^{ij}_l(x+k)=D^{ij}_l(x),\ q_l(x+k)=q_l(x),\ 
f_l(x+k,\cdot)=f_l(x,\cdot)\\ 
&{\rm for}\ \ x\in\R^N,\ k\in \mathbb{L}:=
L_1\Z\times L_2\Z\times\cdots\times L_N\Z\\ 
&(i,j=1,2,\cdots,N;\ l=1,2,\cdots,m),
\end{split}
\end{equation}
and there are two smooth functions 
\[
p^+(x)=(p^+_1(x),p^+_2(x),\cdots,p^+_m(x))\gg 
p^-(x)=(p^-_1(x),p^-_2(x),\cdots,p^-_m(x))
\] 
such that, for some constants $\lambda_{\pm}>0$ and vector valued functions
\[
\varphi^{\pm}={}^t (\varphi^{\pm}_1,\varphi^{\pm}_2,\cdots,\varphi^{\pm}_m)
\gg {}^t (0,0,\cdots,0),
\] 
\begin{equation}\label{cd:x-bistable}
\left\{
\begin{split}
\sum_{i,j=1}^N &D_l^{ij}(x)p^{\pm}_{l,x_ix_j}+q_l(x)\cdot\nabla p^{\pm}_l
+f_l(x,p^{\pm})=0,\\
\sum_{i,j=1}^N &D_l^{ij}(x)\varphi^{\pm}_{l,x_ix_j}+q_l(x)\cdot\nabla \varphi^{\pm}_l
+\sum_{j=1}^m f_{l,u_j}(x,p^{\pm})\varphi^{\pm}_j=-\lambda_{\pm}\varphi^{\pm}_l,\\
&{\rm for }\ \ x\in\R^N\ \ (l=1,2,\cdots,m),\\
p^{\pm}&(\cdot+k)=p^{\pm}(\cdot),\ \varphi^{\pm}(\cdot+k)=\varphi^{\pm}(\cdot)
\ \ {\rm for}\ \ k\in\mathbb{L}.
\end{split}
\right.
\end{equation}
We also assume: 
\begin{equation}\tag{A2}
\left\{
\begin{split}
&\eqref{eq:x-crds}\ {\rm has\ a\ solution}\ 
u(x,t)\ {\rm such\ that}\\
&{\rm for\ a\ constant}\ 
c\not=0\ {\rm and\ a\ unit\ vector}\ n\in\R^N,\\ 
&u(x-k,t)=u(x,t+k\cdot n/c),\ 
\lim_{k\in\mathbb{L},k\cdot n\rightarrow\pm\infty}u(x+k,t)
\rightarrow p^{\mp}(x),\\ 
&c u_{t}(x,t)\gg(0,0,\cdots,0)\ \ 
{\rm for}\ t\in\R,\ x\in\R^N\ {\rm and}\ k\in\mathbb{L}.
\end{split}
\right.
\end{equation}
\begin{remark}
In the case where $m=1$, namely, a scalar bistable equation with spatially periodic coefficients on 
$\R^N$
\begin{equation}\label{eq:x-bistable eq}
u_t-{\rm div}\,(A(x)\nabla u)=F(x,u),\ \ 
x\in \R^N,\ t\in\R,
\end{equation}
Ducrot \cite{D} shows that $({\rm A2})$ is satisfies if and only if there exists no stationary 
front in the direction $n$
under the following assumption: 
\begin{itemize}
\item{} $A:\mathbb{T}^N:=\R^N/\Z^N\rightarrow \mathcal{S}_N$ 
is a symmetric matrix valued function 
of the class $C^{1+\gamma}$ for some $\gamma\in(0,1)$ and satisfies \eqref{cd:ellipticity 1}. 
$F$ is of the class $C^{\gamma}$ in $x$ uniformly with respect to $u\in\R$, 
the partial derivative $F_u$ is continuous on $\mathbb{T}^N\times\R$. 
Moreover the equation
\[
\left\{
\begin{array}{l}\vspace{6pt}
u_t-{\rm div}\,(A(x)\nabla u)=F(x,u),\ \ 
x\in\mathbb{T}^N,\ t>0,\\
u(x,0)=u_0(x),\ \ x\in\mathbb{T}^N
\end{array}
\right.
\]
has two stable stationary states 
$\psi^-<\psi^+$ with $\psi^{\pm}\in C^2(\mathbb{T}^N)$ and there is no 
stable stationary state between $\psi^+$ and $\psi^-$.
\end{itemize}
Fang and Zhao \cite{FZ} give sufficient conditions for $({\rm A2})$ 
in a more abstract framework.  
\end{remark}
\begin{theorem}
[Liouville type theorem for $x$-periodic system]
\label{th:Ltype 3}
Assume $({\rm A2})$, \eqref{cd:ellipticity 1}, \eqref{cd:L-periodic}, \eqref{cd:x-bistable} 
and \eqref{cd:x-cs}. Let $u$ be a solution of \eqref{eq:x-crds} 
and $v$ be a solution as in $({\rm A2})$ with a speed $c\not=0$ and 
a unit vector $n\in\R^N$ which satisfy 
for some constants $a,b\in\R^N$ and 
for all $(x,t)\in\R^N\times\R$,
\begin{equation}\label{eq:sandwich 1}
v(x,t+a)\preceq u(x,t)\preceq v(x,t+b).
\end{equation}
Then $u$ satisfies $({\rm A2})$ with the speed $c\not=0$ and 
the unit vector $n\in\R^N$.
If, in addition, assume \eqref{cd:x-irreducible}, then  
there exists $\theta_0$ between $a$ and $b$ such that, for all $(x,t)\in\R^N\times\R$,
\[
u(x,t)=v(x,t+\theta_0).
\]
\end{theorem}

%%%%%%%%%%%%%%%%%%%%%%%%%%%%%%%%%%%%%%%%%
\section{Proof of the main theorem}
\label{sc:proof of main}
%%%%%%%%%%%%%%%%%%%%%%%%%%%%%%%%%%%%%%%%%
As we mentioned before, the proof of Theorem \ref{th:validity} 
is based on a rescaling argument and the following two 
statements (Lemmas \ref{lm:sandwich} and \ref{lm:Ltype}). 

\begin{lemma}
[\cite{HKMN}]
\label{lm:sandwich}
Let $(u^{\ep},v^{\ep})$ be the solution of \eqref{eq:crds}. 
Under the assumptions in section 1, there are $C>0$, $A_i>0$ $(i=1,2,3)$ 
and $\ep_0>0$ such that, for any $\ep\in(0,\ep_0]$, 
$(x,t)\in\overline{\Omega}\times[Ct^{\ep},T]$, 
\[
\begin{split}
&U_0\Big(\frac{d(x,t)+\ep A_1}{\ep},x\Big)
-\ep A_2
\leq u^{\ep}(x,t)\leq 
U_0\Big(\frac{d(x,t)-\ep A_1}{\ep},x\Big)
+\ep A_2,\\
&V_0\Big(\frac{d(x,t)-\ep A_1}{\ep},x\Big)
-\ep A_3
\leq v^{\ep}(x,t)\leq 
V_0\Big(\frac{d(x,t)+\ep A_1}{\ep},x\Big)
+\ep A_3.
\end{split}
\]
\end{lemma} 

In the next lemma, we consider the following system.
\begin{equation}\label{eq:crds 1}
\left\{
\begin{array}{ll}\vspace{6pt}
u_{t}=D_1\Delta u
+(R_1-a_1u-b_1 v)u,&x\in\R^N,\ t\in\R,\\
\vspace{6pt}
v_{t}=D_2 \Delta v 
+ (R_2-a_2 u-b_2v)v,&x\in\R^N,\ t\in\R,
\end{array}
\right.
\end{equation}
where $D_i$, $R_i$, $a_i$, $b_i$ $(i=1,2)$ are positive constants.
\begin{lemma}
[Liouville type theorem]
\label{lm:Ltype}
Suppose that {\rm Assumption $1$} and \eqref{cd:sc} hold.  
Let $u(x,t)$ $(x\in\R^N,\ t\in\R)$ be a solution of \eqref{eq:crds 1} satisfying, 
for all $(x,t)\in\R^N\times\R$, 
\begin{equation}\label{cd:sandwich}
\left\{
\begin{split}
\phi(n\cdot x-a)\leq u(x,t)\leq \phi(n\cdot x -b),\\
\psi(n\cdot x-b)\leq v(x,t)\leq \psi(n\cdot x -a),
\end{split}
\right.
\end{equation}
where $n$ is a unit vector, $a<b$ are some constants and 
$(\phi,\psi)$ is a solution of \eqref{eq:tw}. 
Then there is a $\theta_0\in(a,b)$ such that,  
for all $(x,t)\in\R^N\times\R$, 
\[
u(x,t)=\phi(n\cdot x - \theta_0),\ \ 
v(x,t)=\psi(n\cdot x - \theta_0).
\]
\end{lemma}
This lemma is a special case of Theorem \ref{th:Ltype}. 
In fact, \eqref{cd:irreducible} is obviously satisfied. 
\eqref{cd:sc} and Assumption 1 imply \eqref{cd:bistable} and (A), respectively.  

\begin{remark}\label{rm:sandwich}
From Lemma \ref{lm:sandwich}, the following holds. (See Theorem 2 in \cite{HKMN}.)
\begin{theorem}\label{th:loc interface}
Let $C>0$, $\ep_0>0$ be constants in Lemma \ref{lm:sandwich}. 
Then there is a constant $\widetilde{C}>0$ 
such that 
\[
d_{\mathcal{H}}(\Gamma^{\ep}(t),\Gamma(t))< \widetilde{C}\ep
\ \ {\rm for}\ \ t\in[Ct^{\ep},T],\ \ep\in(0,\ep_0],
\]
where $d_{\mathcal H}$ denotes the Hausdorff distance between compact sets.
\end{theorem}
\end{remark}

%%%%%%%%%%%%%%%%%%%%%%%%%%%%%%%%%%%%%%%%%
\subsection{Proof of statement (ii)}

\begin{proof}[Poof of $(ii)$ of Theorem \ref{th:validity}] 
Fix $\mu>1$, $T_1\in(T,T_{\max})$ and 
let $C$ be the constant in Lemma \ref{lm:sandwich}. 
To obtain a contradiction, suppose that $(ii)$ does not hold. 
Then there exist $\eta>0$, $\ep_j>0$, $(x_j,t_j)\in\overline{\Omega}\times[\mu Ct^{\ep_j},T]$ 
such that $\ep_j\searrow0$ as $j\rightarrow\infty$ and for all $j\in\N$, 
\[
\left\{
\begin{split}
&\Big|u^{\ep_j}(x_j,t_j)-
U_0\Big(\frac{d^{\ep_j}(x_j,t_j)}{\ep_j},x_j\Big)
\Big|\geq\eta\ \ {\rm or}\\ 
&\Big|v^{\ep_j}(x_j,t_j)-
V_0\Big(\frac{d^{\ep_j}(x_j,t_j)}{\ep_j},x_j\Big)
\Big|\geq\eta.
\end{split}
\right.
\]
By extracting a subsequence, it holds that 
\begin{equation}
\label{eq:contradiction}
\Big|u^{\ep_j}(x_j,t_j)-
U_0\Big(\frac{d^{\ep_j}(x_j,t_j)}{\ep_j},x_j\Big)
\Big|\geq\eta\ \ {\rm for\ all}\ \ j\in\N \ \ {\rm or}\ \ 
\end{equation}
\begin{equation}
\label{eq:contradiction v}
\Big|v^{\ep_j}(x_j,t_j)-
V_0\Big(\frac{d^{\ep_j}(x_j,t_j)}{\ep_j},x_j\Big)
\Big|\geq\eta\ \ {\rm for\ all}\ \ j\in\N.
\end{equation}
Since it is irrelevant in the later argument whether \eqref{eq:contradiction} holds 
or \eqref{eq:contradiction v} holds, 
we may assume that \eqref{eq:contradiction} holds. 
By the same reason, we may assume
\begin{equation}\label{lm:interior}
\begin{split}
&x_k\in\Omega^{\ep_j}_{out}(t_j)\cup\Gamma^{\ep_j}(t_j)\ \ {\rm for\ all}\ \ j\in\N,\ \ 
{\rm that\ is,}\\ 
&(u^{\ep_j}(x_j,t_j),v^{\ep_j}(x_j,t_j))\in\Delta_2\cup S\ \ 
{\rm for\ all}\ \ j\in\N.
\end{split}
\end{equation}
Then it holds that
\begin{equation}\label{eq:loc interface}
{\rm dist}(x_j,\Gamma^{\ep_j}(t_j))= O(\ep_j),\ \ 
{\rm dist}(x_j,\Gamma(t_j))= O(\ep_j)\ \ 
{\rm as}\ \ j\rightarrow\infty.
\end{equation}
In fact, if this is not true, then, by Theorem \ref{th:loc interface} and 
extracting a subsequence, 
it holds that 
\[
\begin{split}
&\Big|\frac{d^{\ep_j}(x_j,t_j)}{\ep_j}\Big|
=\frac{{\rm dist}(x_j,\Gamma^{\ep_j}(t_j))}{\ep_j}\rightarrow\infty
\ \ {\rm as}\ \ j\rightarrow\infty,\\
&\Big|\frac{d(x_j,t_j)}{\ep_j}\Big|
=\frac{{\rm dist}(x_j,\Gamma(t_j))}{\ep_j}\rightarrow\infty
\ \ {\rm as}\ \ j\rightarrow\infty,\\
&d^{\ep_j}(x_j,t_j)d(x_j,t_j)>0\ \ {\rm for\ all}\ \ j\in\N.
\end{split}
\]
By Lemma \ref{lm:sandwich}, 
\[
\begin{split}
0=&\lim_{j\rightarrow\infty}
\Big\{
U_0\Big(\frac{d(x_j,t_j)+\ep_j A_1}{\ep_j},x_j\Big)
-\ep_j A_2
-U_0\Big(\frac{d^{\ep_j}(x_j,t_j)}{\ep_j},x_j\Big)
\Big\}\\
\leq&\lim_{j\rightarrow\infty}
\Big\{
u^{\ep_j}(x_j,t_j)
-U_0\Big(\frac{d^{\ep_j}(x_j,t_j)}{\ep_j},x_j\Big)
\Big\}\\
\leq&\lim_{j\rightarrow\infty}
\Big\{
U_0\Big(\frac{d(x_j,t_j)-\ep_j A_1}{\ep_j},x_j\Big)
+\ep_j A_2
-U_0\Big(\frac{d^{\ep_j}(x_j,t_j)}{\ep_j},x_j\Big)
\Big\}=0
\end{split}
\]
and this contradicts \eqref{eq:contradiction}. Hence \eqref{eq:loc interface} holds. 

Let $y_j\in\Gamma^{\ep_j}(t_j)$ be a point such that 
$|y_j-x_j|=d^{\ep_j}(x_j,t_j)$ and let $p_j=p(x_j,t_j)$ 
be the image of $x_j$ of the projection onto $\Gamma(t_j)$ for each $j\in\N$.
Then it is easy to see that the following hold.
\begin{eqnarray}
\label{eq:1}
&&(u^{\ep_j}(y_j,t_j),v^{\ep_j}(y_j,t_j))\in S,\\
\label{eq:2}
&&d^{\ep_j}(x_j,t_j)=|x_j-y_j|,\\
\label{eq:3}
&&(u^{\ep_j}(x,t_j),v^{\ep_j}(x,t_j))
\in\Delta_2\cup S\ \ {\rm if}\ \ |x-x_j|<|y_j-x_j|,\\
\label{eq:4}
&&x_j-p_j\perp \Gamma(t_j)\ \ {\rm at}\ \ p_j\in\Gamma(t_j),\\
\label{eq:5}
&&|x_j-p_j|=O(\ep_j),\ |x_j-y_j|=O(\ep_j)\ \ 
{\rm as}\ \ j\rightarrow\infty.
\end{eqnarray}
We now rescale the solution $(u^{\ep_j},v^{\ep_j})$ around 
$(x_j,t_j)$ and define 
\begin{equation}\label{eq:rescale}
\left\{
\begin{split}
&w_1^j(z,\tau):=u^{\ep_j}(p_j+\ep_j\mathcal{R}_j z,t_j+\ep_j^2\tau),\\
&w_2^j(z,\tau):=v^{\ep_j}(p_j+\ep_j\mathcal{R}_j z,t_j+\ep_j^2\tau),
\end{split}
\right.
\end{equation}
where $\mathcal{R}_j$ is a matrix in $SO(\R^N)$ that rotates $z_N$ axis 
onto the outward normal at $p_j\in\Gamma(t_j)$. 
Since $\underset{0\leq t\leq T_1}{\cup}\Gamma(t)$ is separated from $\partial\Omega$ 
by some positive distance, there is a $C_0>0$ such that 
$(w_1^j,w_2^j)$ is defined at least on the box 
\[
B_j:=\Big\{(z,\tau)\in\R^N\times\R\mid|z|<\frac{C_0}{\ep_j},\ 
-(\mu-1)C|\log\epsilon_j|\leq t\leq \frac{T_1-T}{\ep_j^2}\Big\}.
\]
Since $(u^{\ep},v^{\ep})$ satisfies \eqref{eq:crds}, we can see that 
$(w^k_1,w^k_2)$ satisfies 
\[
\left\{
\begin{split}
&w_{1,\tau}^j=D_1\tilde{k}_j(z)\Delta w_1^j
+\ep_j q_1^j(z)\cdot\nabla w_1^j+\tilde{h}_j(z)f(w_1^j,w_2^j),\\
&w_{2,\tau}^j=D_2\tilde{k}_j(z)\Delta w_2^j
+\ep_jq_2^j(z)\cdot\nabla w_2^j+\tilde{h}_j(z)g(w_1^j,w_2^j)
\end{split}
\right.\  \ {\rm in}\ \ B_j,
\]
where 
\[
\begin{split}
&\tilde{k}_j(z)=k(p_j+\ep_j\mathcal{R}_j z),\ 
\tilde{h}_j(z)=h(p_j+\ep_j\mathcal{R}_j z),\\
&q_i^j(z)=D_i\nabla k(p_j+\ep_j\mathcal{R}_j z)\ \ (i=1,2).
\end{split}
\]
Thus from \eqref{eq:5}, Lemma \ref{lm:sandwich}, compactness of $\overline{\Omega}$ 
and standard parabolic estimates, 
up to extraction of subsequence, 
$x_j$ and $p_j$ converge to a point $x_*\in\Omega$, 
$(w_1^j,w_2^j)$ converges to $(w_1,w_2)$ locally uniformly in
 $\R^N\times\R=\underset{j\geq1}{\cup}B_j$ as $j\rightarrow\infty$
and 
\[
(w_1(z,\tau),w_2(z,\tau))\ \ {\rm and }\ \ (U_0(z_N,x_*),V_0(z_N,x_*))
\] 
satisfy 
\[
\left\{
\begin{split}
&U_0(z_N+A_1,x_*)
\leq w_1(z,\tau)\leq 
U_0(z_N-A_1,x_*),\\
&V_0(z_N-A_1,x_*)
\leq w_2(z,\tau)\leq 
V_0(z_N+A_1,x_*)
\end{split}
\right.
\]
and the following system
\[
\left\{
\begin{split}
&u_{1,\tau}=D_1k\Delta u_1+h f(u_1,u_2),\\
&u_{2,\tau}=D_2k\Delta u_2+h g(u_1,u_2)
\end{split}
\right.\ \  {\rm in}\ \ \R^N\times\R,
\]
where $k=k(x_*)$, $h=h(x_*)$.
By Lemma \ref{lm:Ltype}, 
there is a $\theta_0\in\R$ such that, for all 
$(z,\tau)\in\R^N\times\R$, 
\begin{equation}\label{eq:w=tw}
(w_1(z,\tau),w_2(z,\tau))=(U_0(z_N-\theta_0,x_*),V_0(z_N-\theta_0,x_*)).
\end{equation}
Define 
\[
z_j:=\frac{1}{\ep_j}\mathcal{R}_j^{-1}(x_j-p_j),\ \ 
\tilde{z}_j:=\frac{1}{\ep_j}\mathcal{R}_j^{-1}(y_j-p_j)\ \ 
(j\in\N).
\]
From \eqref{eq:5}, up to extraction of subsequence, they converge:
\[
\lim_{j\rightarrow\infty}z_j=z_*=(0,\cdots,0,z_{*,N}),\ \ 
\lim_{j\rightarrow\infty}\tilde{z}_j=\tilde{z}_*
=(\tilde{z}_{*,1},\cdots,\tilde{z}_{*,N}).
\]
By \eqref{eq:1} and \eqref{eq:3}, 
\begin{equation}\label{eq:6}
\begin{split}
&(w_1(\tilde{z}_*,0),w_2(\tilde{z}_*,0))\in S,\\
&(w_1(z,0),w_2(z,0))\in \Delta_2\cup S
\ \ {\rm for}\ \ z\ \ {\rm with}\ \ |z-z_*|\leq|\tilde{z}_*-z_*|.
\end{split}
\end{equation}
By \eqref{eq:w=tw}, 
\[
\{z\mid (w_1(z,0),w_2(z,0))\in S\}=\{z\mid z_N=\theta_0\}\,(=:H).
\]
By \eqref{eq:w=tw} and \eqref{eq:6}, 
\[
z_*=\tilde{z}_*\ \ {\rm or}\  \ 
\partial B_{|\tilde{z}_*-z_*|}(z_*)\ \ {\rm and}
\ \ H\ \ {\rm intersect\ at}\ \ \tilde{z}_*.
\]
Thus 
\[
\tilde{z}_*=(0,\cdots,0,\tilde{z}_{*,N})=(0,\cdots,0,\theta_0).
\]
By $(w_1(z_*,0),w_2(z_*,0))\in\Delta_2\cup S$ and \eqref{eq:w=tw}, 
\[
z_{*,N}\geq\theta_0.
\] 
On the other hand, 
$d^{\ep_j}(x_j,t_j)=|y_j-x_j|$ implies 
\[
\frac{d^{\ep_j}(x_j,t_j)}{\ep_j}=|z_j-\tilde{z}_j|\rightarrow|z_*-\tilde{z}_*|=z_{*,N}-\theta_0
\ \ {\rm as}\ \ j\rightarrow\infty.
\]
Hence, by \eqref{eq:contradiction} and \eqref{eq:w=tw}, 
\[
\begin{split}
0&=|w_1(z_*,0)-U_0(z_{*,N}-\theta_0,x_*)|\\
&=\lim_{j\rightarrow\infty}
\Big|u^{\ep_j}(x_j,t_j)
-U_0\Big(\frac{d^{\ep_j}(x_j,t_j)}{\ep_j},x_j\Big)\Big|
\geq\eta>0.
\end{split}
\]
This contradiction proves that (ii) of Theorem \ref{th:validity} holds. 
\end{proof}
%%%%%%%%%%%%%%%%%%%%%%%%%%%%%%%%%%%%%%%%%
\subsection{Proof of statements (i) and (iii)}

\begin{proof}[Proof of (i), (iii) of Theorem \ref{th:validity}]
First we prove that there is a constant $c_1>0$ such that 
for all $x\in\mathcal{N}_{\tilde{C}\ep}(\Gamma(t))$, 
$t\in[\mu Ct^{\ep},T]$ and $\ep\in(0,\ep_0]$,
\begin{equation}\label{eq:non degenerateness 1}
\begin{split}
H_u(&u^{\ep}(x,t),v^{\ep}(x.t))\,n(p(x,t),t)\cdot\nabla u^{\ep}(x,t)\\
&+H_v(u^{\ep}(x,t),v^{\ep}(x.t))\,n(p(x,t),t)\cdot\nabla v^{\ep}(x,t)
\geq \frac{c_1}{\ep},
\end{split}
\end{equation}
\begin{equation}\label{eq:degenerateness}
\lim_{\ep\rightarrow0}
\sup_{t\in[\mu Ct^{\ep},T]}
\{\ep\|\nabla_{\Gamma(t)} u^{\ep}(\cdot,t)\|_{L^{\infty}(\mathcal{N}_{\tilde{C}\ep}(\Gamma(t)))}
+\ep\|\nabla_{\Gamma(t)} v^{\ep}(\cdot,t)\|_{L^{\infty}(\mathcal{N}_{\tilde{C}\ep}(\Gamma(t)))}\}
=0,
\end{equation}
where $C$, $\tilde{C}$ and $\ep_0$ are constants 
in Lemma \ref{lm:sandwich} and in Theorem \ref{th:loc interface}, respectively 
and $H(u,v)$ is a function in Remark \ref{rm:S and tw}, 
$n(p,t)$ is the outward unit normal vector to $\Gamma(t)$ at $p\in\Gamma(t)$, 
$p(x,t)$ is the image of $x$ of the projection onto $\Gamma(t)$ and 
\[
\mathcal{N}_{\tilde{C}\ep}(\Gamma(t))
:=\{x\mid {\rm dist}(x,\Gamma(t))<\tilde{C}\ep\}.
\] 
If \eqref{eq:non degenerateness 1} is not true, 
then there exist $\ep_j>0$, $t_j\in[\mu Ct^{\ep_j},T]$ and 
$x_j\in\mathcal{N}_{\tilde{C}\ep_j}(\Gamma(t_j))$ such that 
\begin{equation}\label{eq:contradiction 1}
\begin{split}
\lim_{j\rightarrow\infty}
&\ep_j\{H_u(u^{\ep_j}(x_j,t_j),v^{\ep_j}(x_j.t_j))\,n(p(x_j,t_j),t_j)\cdot\nabla u^{\ep_j}(x_j,t_j)\\
&+H_v(u^{\ep_j}(x_j,t_j),v^{\ep_j}(x_j.t_j))\,n(p(x_j,t_j),t_j)\cdot\nabla v^{\ep_j}(x_j,t_j)\}
=0.
\end{split}
\end{equation}
By the same rescaling argument as in the proof of the statement (ii), 
the rescaled function $(w_1^j(z,\tau),w_2^j(z,\tau))$ 
converges to $(U_0(z_N-\theta_0,x_*),V_0(z_N-\theta_0,x_*))$ in $C^{2,1}_{loc}(\R^N\times\R)$ 
as $j\rightarrow\infty$ and 
\[
\begin{split}
H_u(U_0(&-\theta_0,x_*),V_0(-\theta_0,x_*))\, U_0'(-\theta_0,x_*)\\
&+H_v(U_0(-\theta_0,x_*),V_0(-\theta_0,x_*))V'_0(-\theta_0,x_*)\\
=&\lim_{j\rightarrow\infty}
\{H_u(w_1^j(0,0),w_2^j(0,0))\,n(p(x_j,t_j),t_j)\cdot\nabla w_1^j(0,0)\\
&+H_v(w_1^j(0,0),w_2^j(0,0))\,n(p(x_j,t_j),t_j)\cdot\nabla w_2^j(0,0)\}
=0.
\end{split}
\]
This contradicts \eqref{eq:non degenerateness} in Remark \ref{rm:S and tw} and 
this contradiction implies 
that \eqref{eq:non degenerateness 1} holds. The proof of \eqref{eq:degenerateness} 
is similar to that of \eqref{eq:non degenerateness 1} and we omit it.

By \eqref{eq:non degenerateness}, Theorem \ref{th:loc interface}, 
\eqref{eq:non degenerateness 1} and the implicit function theorem, 
there is a smooth function $\eta^{\ep}(\cdot,t)$ 
defined on $\Gamma(t)$ for each $t\in[\mu Ct^{\ep},T]$ 
such that 
\begin{eqnarray}
\label{eq:implicit 1}
&&H(u^{\ep}(x+\eta^{\ep}(x,t)n(x,t),t),v^{\ep}(x+\eta^{\ep}(x,t)n(x,t),t))=0,\\
\label{eq:implicit 2}
&&H(u^{\ep}(y,t),v^{\ep}(y,t))=0
\Leftrightarrow
\exists x\in\Gamma(t),\ {\rm s.t.},\ y=x+\eta^{\ep}(x,t)n(x,t),\\
\label{eq:gradient}
&&\nabla_{\Gamma(t)}\eta^{\ep}(x,t)
=-\frac{H_u
\nabla_{\Gamma(t)}u^{\ep}(p^{\ep},t)
+H_v
\nabla_{\Gamma(t)}v^{\ep}(p^{\ep},t)}{
H_u
\,n^{\ep}\cdot
\nabla u^{\ep}(p^{\ep},t)
+H_v
\,n^{\ep}\cdot
\nabla v^{\ep}(p^{\ep},t)}\\
&&\ \ {\rm for\ all}\ \ x\in\Gamma(t),\ t\in[\mu Ct^{\ep},T],\nonumber
\end{eqnarray}
where $\nabla_{\Gamma(t)}$ denotes the gradient on $\Gamma(t)$, 
$n^{\ep}=n(p^{\ep},t)$, 
\[
\begin{split}
&p^{\ep}=p^{\ep}(x,t)=x+\eta^{\ep}(x,t)n(x,t),\\
&H_u=H_u(u^{\ep}(p^{\ep},t),v^{\ep}(p^{\ep},t)),\ 
H_v=H_v(u^{\ep}(p^{\ep},t),v^{\ep}(p^{\ep},t)).
\end{split}
\]
From \eqref{eq:implicit 1} and \eqref{eq:implicit 2}, it holds that 
$\Gamma^{\ep}(t)$ is expressed as the graph of the function $\eta^{\ep}(\cdot,t)$ 
on $\Gamma(t)$ for each $t\in[\mu Ct^{\ep},T]$. 
By \eqref{eq:non degenerateness 1}, \eqref{eq:degenerateness} and \eqref{eq:gradient}
\[
\begin{split}
|\nabla_{\Gamma(t)}\eta^{\ep}(x,t)|
&= O(\ep
|\nabla_{\Gamma(t)}u^{\ep}(p^{\ep},t)|+\ep|\nabla_{\Gamma(t)}v^{\ep}(p^{\ep},t)|)
\rightarrow0\\ 
&{\rm as}\ \ \ep\rightarrow0\ \ 
{\rm uniformly\ for}\ \ x\in\Gamma(t),\ t\in[\mu Ct^{\ep},T].
\end{split}
\]
This completes the proof of statement (i). 
Statement (iii) immediately follows from (ii) and 
\[
\sup_{x\in\Omega,\,t\in[\mu Ct^{\ep},T]}
|d^{\ep}(x,t)-d(x,t)|\leq \tilde{C}\epsilon\ \ {\rm for\ all}\ \ \ep\in(0,\ep_0].
\]
\end{proof}
%%%%%%%%%%%%%%%%%%%%%%%%%%%%%%%%%%%%%%%%%
\subsection{Proof of Lemmas \ref{lm:sandwich}\ and \ref{lm:Ltype}}

\begin{proof}[Proof of Lemma \ref{lm:sandwich}]
In Definition 6.1 in \cite{HKMN}, they construct a lower solution 
$(\hat{u}^-_{\ep},\hat{v}^-_{\ep})$ 
and a upper solution $(\hat{u}^+_{\ep},\hat{v}^+_{\ep})$ of \eqref{eq:crds}. 
By $(7.1)$ in \cite{HKMN} and their construction, it is easy to see that there are 
$C>0$, $A_i>0$ $(i=1,2,3)$ and $\ep_0>0$ such that, 
for any $\ep\in(0,\ep_0]$, 
$(x,t)\in\overline{\Omega}\times[Ct^{\ep},T]$, 
\[
\begin{split}
U_0\Big(\frac{d(x,t)+\ep A_1}{\ep},x\Big)
&-\ep A_2\\
\leq\hat{u}^-_{\ep}(x,t)
&\leq u^{\ep}(x,t)\leq\hat{u}^+_{\ep}(x,t)\\
&\leq U_0\Big(\frac{d(x,t)-\ep A_1}{\ep},x\Big)
+\ep A_2,\\
V_0\Big(\frac{d(x,t)-\ep A_1}{\ep},x\Big)
&-\ep A_3\\
\leq\hat{v}^+_{\ep}(x,t)
&\leq v^{\ep}(x,t)\leq \hat{v}^-_{\ep}(x,t)\\
&\leq V_0\Big(\frac{d(x,t)+\ep A_1}{\ep},x\Big)
+\ep A_3.
\end{split}
\]
This completes the proof.
\end{proof}
\begin{proof}[Proof of Lemma \ref{lm:Ltype}]
This lemma is an easy consequence of Theorem \ref{th:Ltype}. 
In fact, it is obvious that \eqref{cd:irreducible} holds. 
\eqref{cd:bistable} and $({\rm A})$ follow from \eqref{cd:sc} and {\rm Assumption $1$}, 
respectively.
\end{proof}
%%%%%%%%%%%%%%%%%%%%%%%%%%%%%%%%%%%%%%%%%
\section{Proof of the Liouville type theorems}\label{sc:proof of Ltype}

%%%%%%%%%%%%%%%%%%%%%%%%%%%%%%%%%%%%%%%%%
\subsection{Proof of Theorems \ref{th:Ltype}\ and \ref{th:Ltype 1}}

\begin{proof}[Proof of Theorem \ref{th:Ltype}]
Let us put 
\[
(u_1,u_2)=(u,-v),\ f_1(u_1,u_2)=f(u_1,-u_2),\ f_2(u_1,u_2)=-g(u_1,-u_2).
\] 
Then Theorem \ref{th:Ltype} is an easy consequence of Theorem \ref{th:Ltype 1} and 
we omit the detail of the proof. 
\end{proof}
Next proposition plays a key rule to prove the uniqueness of the traveling wave solution 
up to shifts in time.
\begin{proposition}[strong comparison principle]\label{pr:scp}
Assume \eqref{cd:bistable 1} and \eqref{cd:cs 1}. 
Let $u(x,t)$, $v(x,t)$ be solutions of \eqref{eq:crds 3} such that 
\[
p^-\preceq u,\,v\preceq p^+,\ \ 
u(\cdot,0)\preceq v(\cdot,0)
\]
Then 
$u(\cdot,t)\preceq v(\cdot,t)$ for any $t\geq0$. 
If, in addition, assume \eqref{cd:irreducible 1} and $u(\cdot,0)\prec v(\cdot,0)$, then 
$u(\cdot,t)\ll v(\cdot,t)$ for any $t>0$.
\end{proposition}
\begin{proof}
First we prove 
\begin{equation}\label{eq:wcp}
u(\cdot,0)\ll v(\cdot,0)\Rightarrow
u(\cdot,t)\ll v(\cdot,t)\ \ {\rm for\ all}\ \ t\geq0.
\end{equation}
If \eqref{eq:wcp} does not hold, then 
\[
t_0:=\sup\{t'>0\mid u(\cdot,t)\ll v(\cdot,t)\ \ {\rm for\ all}\ \ t\in[0,t']\}
\in(0,\infty)
\]
and 
\[
\begin{split}
&u(\cdot,t)\ll v(\cdot,t)\ \ {\rm for}\ \ t\in[0,t_0),\\
&u(\cdot,t_0)\preceq v(\cdot,t_0),\ u(\cdot,t_0)\not\ll v(\cdot,t_0).
\end{split}
\]
Hence $\zeta(x,t):=v(x,t)-u(x,t)$ satisfies 
\[
\begin{split}
&\zeta(\cdot,t)\gg(0,0,\cdots,0)\ \ {\rm for}\ \ t\in[0,t_0),\ \ 
\zeta(\cdot,t_0)\succeq(0,0,\cdots,0),\\
&\exists l_0\in\{1,2,\cdots,m\},\,x_0\in\R^N,\ \zeta_{l_0}(x_0,t_0)=0.
\end{split}
\]
By \eqref{cd:cs 1}, for $(x,t)\in\R^N\times(0,t_0]$, 
\begin{equation}\label{eq:PDIneq}
\begin{split}
\zeta_{l_0,t}&-D_{l_0}\Delta \zeta_{l_0}=f_{l_0}(v)-f_{l_0}(u)\\
&\geq f_{l_0}(\cdots,u_{l_0-1},v_{l_0},u_{l_0+1},\cdots)-f_{l_0}(u)\\
&\geq-M\zeta_{l_0}\ \ \ (M:=\sup_{w\in[p^-,p^+]}|DF(w)|).
\end{split}
\end{equation}
By strong maximum principle for a parabolic equation, 
\[
\zeta_{l_0}(\cdot,t)=0\ \ {\rm for }\ \ t\in[0,t_0].
\] 
This contradicts $\zeta(\cdot,t)\gg(0,0,\cdots,0)$ for $t\in[0,t_0)$. 
Hence \eqref{eq:wcp} holds. We take 
some smooth functions $u^j(\cdot,0)$, $v^j(\cdot,0)$ such that 
\[
\begin{split}
&\lim_{j\rightarrow\infty}u^j(\cdot,0)=u(\cdot,0),\ 
\lim_{j\rightarrow\infty}v^j(\cdot,0)=v(\cdot,0),\\
&p^-\ll u^j(\cdot,0)\ll v^j(\cdot,0)\ll p^+\ \ {\rm for}\ \ j\in\N.
\end{split}
\]
and let $u^j$, $v^j$ be solutions of \eqref{eq:crds 3} with initial data $u^j(\cdot,0)$, $v^j(\cdot,0)$, 
respectively. Then by \eqref{eq:wcp}, 
\[
u^j(\cdot,t)\ll v^j(\cdot,t)\ \ {\rm for}\ \ t\geq0.
\]
Taking limits of both sides of this inequality as $j\rightarrow\infty$, 
by continuously dependence of solutions of 
a parabolic system on initial data, 
we get 
\[
u(\cdot,t)\preceq v(\cdot,t)\ \ {\rm for}\ \ t\geq0.
\]
Now we assume \eqref{cd:irreducible 1} and prove 
\[
u(\cdot,0)\prec v(\cdot,0)\Rightarrow u(\cdot,t)\ll v(\cdot,t)
\ \ {\rm for}\ \ t>0.
\]
If this is not true, then there are $l_0\in\{1,2,\cdots,m\}$ and 
$(x_0,t_0)\in\R^N\times(0,\infty)$ 
such that $u_{l_0}(x_0,t_0)=v_{l_0}(x_0,t_0)$. 
$\zeta_{l_0}:=u_{l_0}-v_{l_0}$ satisfies \eqref{eq:PDIneq},  
$\zeta_{l_0}\geq0$ and $\zeta_{l_0}(x_0,t_0)=0$. By strong maximum principle, 
$\zeta_{l_0}=0$ for $(x,t)\in\R^N\times[0,t_0]$. 
Put
\[
\Lambda:=\{l\in\{1,2,\cdots,m\}\mid \zeta_{l}:=u_{l}-v_{l}\not>0\ \ {\rm on}\ \ \R^N\times\{t_0\}\}
\]
Then by $u(\cdot,0)\prec v(\cdot,0)$, 
\[
\emptyset\not=
\Lambda\subsetneq\{1,2,\cdots,m\}
\] 
and 
\[
\left\{
\begin{split}
&\zeta_l(\cdot,t)=0\ \ {\rm in}\ \ \R^N\times[0,t_0]
\ \ {\rm for}\ \ l\in\Lambda,\\
&v_l(\cdot,t_0)-u_l(\cdot,t_0)=\zeta_l(\cdot,t_0)>0\ \ {\rm in}\ \ \R^N
\ \ {\rm for}\ \ l\not\in\Lambda.
\end{split}
\right.
\]
Then 
\begin{equation}\label{eq:reducible}
f_{l,u_j}(v(\cdot,t_0))=0\ \ {\rm for}\ \ l\in\Lambda,\ 
j\not\in\Lambda.
\end{equation}
If \eqref{eq:reducible} is not true, then  
\[
f_{l_0,u_{j_0}}(v(x_0,t_0))>0\ \ {\rm for\ some}\ \ x_0\in\R^N\ \ 
{\rm and}\ \ l_0\in\Lambda,\ j_0\not\in\Lambda
\] 
and hence, 
by $\zeta_{l_0}=0$ in $\R^N\times[0,t_0]$ and $v_{j_0}(x_0,t_0)>u_{j_0}(x_0,t_0)$, 
\[
\begin{split}
0=\zeta_{l_0,t}&-D_{l_0}\Delta \zeta_{l_0}=f_{l_0}(v)-f_{l_0}(u)\\
&> f_{l_0}(\cdots,v_{j_0-1},u_{j_0},v_{j_0+1},\cdots)-f_{l_0}(u)\\
&\geq f_{l_0}(\cdots,u_{l_0-1},v_{l_0},u_{l_0+1},\cdots)-f_{l_0}(u)\\
&\geq-M\zeta_{l_0}=0\ \ {\rm at}\ \ (x,t)=(x_0,t_0).
\end{split}
\]
This is a contradiction and \eqref{eq:reducible} holds. 
However \eqref{eq:reducible} implies that $DF(v(x,t_0))$ is reducible for 
each $x\in\R^N$ and contradicts the assumption \eqref{cd:irreducible 1}. 
The proof is completed. 
\end{proof}
Next lemma completes the proof of the last part of Theorem \ref{th:Ltype 1}. 
We give the proof later. 
\begin{lemma}\label{lm:uniqueness}
Assume \eqref{cd:bistable 1}, \eqref{cd:cs 1} and \eqref{cd:irreducible 1}. 
Let $\phi(n\cdot x-ct), \widetilde{\phi}(n\cdot x-ct)$ be functions satisfying $({\rm A'})$ 
with a direction $n\in\R^N$ and a speed $c$ and for some constants $a,b\in\R$,  
\begin{equation}\label{eq:sandwiched}
\widetilde{\phi}(\cdot-a)\preceq\phi(\cdot)\preceq\widetilde{\phi}(\cdot-b).
\end{equation}
Then $\phi(\cdot)=\widetilde{\phi}(\cdot-\theta_0)$ 
for some $\theta_0\in\R$. 
\end{lemma}

\begin{proof}
[Proof of Theorem \ref{th:Ltype 1}]
Take 
\[
0<\delta<
\min\Big\{\frac{p^+_l-p^-_l}{\max\{\varphi^-_l,\varphi^+_l\}}\mid l=1,2,\cdots,m\Big\}
\] 
such that 
\begin{equation}\label{cd:delta}
\delta<\frac{\min\{\lambda_{+}\varphi_l^{+},\,\lambda_{-}\varphi_l^{-}\mid
l=1,2,\cdots,m\}}{\max\Big\{1,\,
\underset{u\in[p^-,p^+]}{\sup}|D^2F(u)|\Big\}},
\end{equation}
where 
$[p^-,p^+]:=\{u\in\R^m\mid p^-_l\leq u_l\leq p^+_l,\ \ l=1,2,\cdots,m\}$ 
and
\[
|D^2F|:=\sqrt{\sum_{l,j,k=1}^m f^2_{l,u_j u_k}}.
\]
Then 
\begin{equation}\label{eq:near p^+}
\begin{split}
F(w)\preceq &F(w-\ep\varphi^+)
\ \ {\rm for\ any}\ \ w\in\R^m\ \ {\rm with}\\
&p^+\succeq w\succeq p^+ - \frac{\delta}{2}\varphi^+
\ \ {\rm and\ for\ any}\ \ \ep\in[0,\delta/2],
\end{split}
\end{equation}
\begin{equation}\label{eq:near p^-}
\begin{split}
F(w)\succeq &F(w+\ep\varphi^-)
\ \ {\rm for\ any}\ \ w\in\R^m\ \ {\rm with}\\
&p^-\preceq w\preceq p^- +\frac{\delta}{2}\varphi^-
\ \ {\rm and\ for\ any}\ \ \ep\in[0,\delta/2].
\end{split}
\end{equation}
In fact, for any $w\in\R^m$ with 
$
p^+\succeq w\succeq p^+ - \frac{\delta}{2}\varphi^+
$
and $\ep\in[0,\delta/2]$, by using \eqref{cd:bistable 1},
\[
\begin{split}
F(w)&=F(w-\ep\varphi^+)+F(w)-F(w-\ep\varphi^+)
- \ep DF(p^+)\varphi^+ - \ep\lambda_+\varphi^+\\
&=F(w-\ep\varphi^+)+\ep\int_0^1\{DF(w-s\epsilon\varphi^+)-DF(p^+)\}\varphi^+ ds
- \ep\lambda_+\varphi^+.
\end{split}
\]
By \eqref{cd:delta}, $\ep\in[0,\delta/2]$ 
and ${}^t (0,0,\cdots,0)\succeq w-p^+\succeq -(\delta/2)\varphi^+$, 
\[
\begin{split}
\Big|\int_0^1&\{DF(w-s\epsilon\varphi^+)
-DF(p^+)\}\varphi^+ ds\Big|
\leq(\ep+\delta/2)\sup_{u\in[p^-,p^+]}|D^2F(u)|\\
&
\leq\delta\sup_{u\in[p^-,p^+]}|D^2F(u)|
\leq \lambda_+\varphi^+_l\ \ \ \ {\rm for}\ \ l=1,2,\cdots,m.\\
\therefore\epsilon\int_0^1&\{DF(w-s\epsilon\varphi^+)
-DF(p^+)\}\varphi^+ ds
\preceq \ep\lambda_+\varphi^+.
\end{split}
\]
Thus \eqref{eq:near p^+} holds. 
The proof of \eqref{eq:near p^-} is similar and we omit it.

By \eqref{cd:cs 1},
\begin{equation}\label{eq:competition}
\begin{split}
&f_l(v)-f_l(u)\geq -M\,(v_l-u_l)
\  \ (l=1,2,\cdots,m)\\
&\ \ {\rm for\ any}\ \ u,\,v\in[p^-,p^+]\ \ 
{\rm with}\ \ u\preceq v\ \ \ \ (M:=\max_{w\in[p^-,p^+]}|DF(w)|).
\end{split}
\end{equation}
Taking appropriate Cartesian coordinates, we may assume that 
the solution $u(x,t)$ of \eqref{eq:crds 3} satisfies 
\begin{equation}\label{cd:sandwich}
\phi(x_1 - c t - a) \preceq u(x,t) \preceq \phi(x_1 - c t - b)
\end{equation}
for all $(x,t)\in\R^N\times\R$. 

Fix $(\rho,\tau)\in\R^{N-1}\times\R$ arbitrary. For $\sigma\in\R$, 
put 
\[
w^{\sigma}(x,t):=u(x_1+c\tau+\sigma,x'+\rho,t+\tau)
\] 
for all $(x,t)\in\R^N\times\R$. 
By \eqref{cd:sandwich} and monotonicity of $\phi(z)$, 
for any $\sigma\geq b-a$, 
\[
w^{\sigma}(x,t)\preceq \phi(x_1+\sigma-ct-b)
\preceq \phi(x_1-ct-a)\preceq u(x,t)
\]
for all $(x,t)\in\R^N\times\R.$
Define 
\[
\sigma^*:=\inf\{\sigma\mid w^{\sigma'}\preceq u
\ \  {\rm holds\ for\ all}\ \ \sigma'\geq\sigma\}.
\]
Then clearly $\sigma^*\leq b-a$. 

Now we prove $\sigma^*\leq0$ by contradiction. 
Suppose $\sigma^*>0$. By \eqref{cd:sandwich}, 
monotonicity of $\phi(z)$ and $\phi(\pm\infty)=p^{\mp}$, 
there is $C>b-a$ such that 
\begin{equation}\label{eq:asymptotic}
\left\{
\begin{split}
&p^-\preceq u(x_1,x',t)\preceq p^- + \frac{\delta}{2} \varphi^-
\\ &\hspace{40pt}{\rm for\ all}\ \ x_1-ct\geq C,\ (x',t)\in\R^{N-1}\times\R,\\
&p^+\succeq u(x_1,x',t)\succeq p^+ - \frac{\delta}{2} \varphi^+
\\ &\hspace{40pt}{\rm for\ all}\ \ x_1-ct\leq -C,\ (x',t)\in\R^{N-1}\times\R,
\end{split}
\right.
\end{equation}
where 
\[
0<\delta<
\min\Big\{\frac{p^+_l-p^-_l}{\max\{\varphi_l^+,\varphi^-_l\}}\mid l=1,2,\cdots,m\Big\}
\] 
is a constant for which \eqref{cd:delta} holds. 

If 
$
(0,0,\cdots,0)\not\ll\inf\{u-w^{\sigma^*}\mid |x_1-ct|\leq 2C,\ (x',t)\in\R^{N-1}\times\R\},
$
then there exist $l_0\in\{1,2,\cdots,m\}$, $x_{1,\infty}\in[-2C,2C]$ and 
\[
x_{1,n}-ct_n\in[-2C,2C],\ (x'_{n},t_n)\in\R^{N-1}\times\R
\ \ (n=1,2,3,\cdots)
\] 
such that 
\[
u_{l_0}(x_{1,n},x'_n,t_n)-w^{\sigma^*}_{l_0}(x_{1,n},x'_n,t_n)\rightarrow0,\ \ 
x_{1,n}-ct_n\rightarrow x_{1,\infty}\ \ {\rm as}\ \ n\rightarrow\infty.
\]
From standard parabolic estimates, up to extraction of subsequence, 
the function $u^n(x,t):=u(x_1+ct_n,x'+x_n',t+t_n)$ 
converges locally uniformly to a solution $U$ 
of \eqref{eq:crds 3} such that 
\[
\begin{split}
z(x,t)&:=U(x,t)-W^{\sigma^*}(x,t)\\
&:=U(x,t)-U(x_1+c\tau+\sigma^*,x'+\rho,t+\tau)\succeq(0,0,\cdots,0),\\
z_{l_0}(x_{1,\infty},&0,0)=0
\end{split}
\]
and by \eqref{eq:competition} and $W^{\sigma^*}\preceq U$, 
\[
z_{l_0,t}-D_{l_0}\Delta z_{l_0}=f_{l_0}(U)
-f_{l_0}(W^{\sigma^*})
\geq-M z_{l_0}
\]
for all $(x,t)\in \R^N\times\R$. 
By strong maximum principle, for all $(x,t)\in\R^N\times(-\infty,0]$, 
\[
\begin{split}
U_{l_0}(x,t)-U_{l_0}(x_1+c\tau+\sigma^*,x'+\rho,t+\tau)&=
U_{l_0}(x,t)-W^{\sigma^*}_{l_0}(x,t)\\
&=z_{l_0}(x,t)=0.
\end{split}
\]
If $\tau>0$, then, 
by 
\[
\begin{split}
&\sigma^*>0,\ \phi(-\infty)=p^+,\ \phi(+\infty)=p^-\ \ {\rm and} \\
&\phi(x_1-ct-a)\preceq U(x,t)\preceq \phi(x_1-ct-b),
\end{split}
\]
\[
\begin{split}
U_{l_0}(x,0)&=U_{l_0}(x_1-c\tau-\sigma^*,x'-\rho,-\tau)\\
&=\cdots=U_{l_0}(x_1-cn\tau-n\sigma^*,x'-n\rho,-n\tau)
\overset{n\rightarrow\infty}{\longrightarrow} p^+_{l_0}
\end{split}
\]
and this contradicts 
\[
U(x,0)\preceq \phi(x_1-b)\ll p^+.
\]
If $\tau\leq0$, then 
\[
\begin{split}
U_{l_0}(x,0)&=U_{l_0}(x_1+c\tau+\sigma^*,x'+\rho,\tau)\\
&=\cdots=U_{l_0}(x_1+cn\tau+n\sigma^*,x'+n\rho,n\tau)
\overset{n\rightarrow\infty}{\longrightarrow} p^-_{l_0}
\end{split}
\]
and this contradicts 
\[
p^-\ll\phi(x_1-a)\preceq U(x,0).
\]
Then it follows that 
\[
(0,0,\cdots,0)\ll\inf\{u-w^{\sigma^*}\mid (x,t)\in\R^N\times\R,\ 
|x_1-ct|\leq 2C\}.
\]
Hence, by uniformly continuity of $u$, 
there is an $\eta_0\in(0,\sigma^*)$ such that, for any $\eta\in[0,\eta_0]$, 
\begin{equation}\label{eq:inequality}
(0,0,\cdots,0)\ll
\inf\{u-w^{\sigma^*-\eta}\mid \ (x,t)\in\R^{N}\times\R,\ |x_1-ct|\leq 2C\}.
\end{equation}
By $u\succeq w^{\sigma^*}$, $\varphi^{\pm}\gg {}^t(0,0,\cdots,0)$ and uniformly continuity of $u$, 
there is an $\eta_1\in(0,\eta_0]$ such that, for any $\eta\in[0,\eta_1]$, 
\[
u+\frac{\delta}{2}\varphi^{\pm}
\succeq w^{\sigma^*}+\frac{\delta}{2}\varphi^{\pm}
\succeq w^{\sigma^*-\eta}.
\]
Put $S_{\pm}:=\{(x,t)\in\R^N\times\R\mid \pm (x_1-ct)\leq -2C\}$ and 
\[
\ep_{\pm}:=\inf\{\ep>0\mid u+\ep\varphi^{\pm}
\succeq w^{\sigma^*-\eta}\ {\rm on}\ S_{\pm}\}\ (\in[0,\delta/2]).
\]
We prove $\ep_{\pm}=0$ by contradiction. Suppose $\ep_{\pm}>0$. 
Then, by 
\[
u+\ep_{\pm}\varphi^{\pm}
-w^{\sigma^*-\eta}\succeq\epsilon_{\pm}\varphi^{\pm}\gg(0,0,\cdots,0)
\ \ {\rm on}\ \ \partial S_{\pm}
\] 
and 
\[
u(\mp\infty,x',t)+\ep_{\pm}\varphi^{\pm}
-w^{\sigma^*-\eta}(\mp\infty,x',t)=\ep_{\pm}\varphi^{\pm}\gg(0,0,\cdots,0),
\]
for each $\mu\in\{+,-\}$, 
there exist $l_{\mu}\in\{1,2\}$, $\mu x_{1,\infty}^{\mu}\in(-\infty,-2C)$ and 
\[
(x_n^{\mu},t_n^{\mu})=(x_{1,n}^{\mu},x_{n}^{\mu\prime},t_n^{\mu})\in S_{\mu}
\ \ (n=1,2,3,\cdots)
\] 
such that 
\[
u_{l_{\mu}}(x_n^{\mu},t_n^{\mu})+\ep_{\mu}\varphi^{\mu}
-w^{\sigma^*-\eta}_{l_{\mu}}(x_n^{\mu},t_n^{\mu})\rightarrow0,\ \ 
x^{\mu}_{1,n}-ct_n^{\mu}\rightarrow x^{\mu}_{1,\infty}\ \ {\rm as}\ \ 
n\rightarrow\infty.
\]
From standard parabolic estimates, up to extraction of subsequence, 
the functions $u^{\mu,n}(x,t):=u(x_1+ct_n^{\mu},x'+x_n^{\mu \prime},t+t_n^{\mu})$ 
converge locally uniformly to solutions $U^{\mu}$ of \eqref{eq:crds 3} such that, 
for $(x,t)\in S_{\mu}$, 
\[
\begin{split}
z^{\mu}(x,t)&:=U^{\mu}(x,t)+\ep_{\mu}\varphi^{\mu}
-W^{\mu,\sigma^*-\eta}(x,t)\\
&:=U^{\mu}(x,t)+\ep_{\mu}\varphi^{\mu}
-U^{\mu}(x_{1}+c\tau+\sigma^*-\eta,x'+\rho,t+\tau)\\
&\succeq(0,0,\cdots,0),\\
z^{\mu}_{l_{\mu}}(x_{1,\infty}^{\mu}&,0,0)=0,\ \ 
z^{\mu}\succeq\epsilon_{\mu}\varphi^{\mu}\gg(0,0,\cdots,0)\ \ {\rm on}\ \ \partial S_{\mu}
\end{split}
\]
and, by \eqref{eq:asymptotic}, 
\[
\begin{split}
&p^+\succeq W^{+,\sigma^*-\eta}\succeq p^+ -\frac{\delta}{2}\varphi^+
\ \ {\rm on}\ \ S_+,\\
&p^-\preceq U^{-}\preceq p^-+\frac{\delta}{2}\varphi^-
\ \ {\rm on}\ \ S_-
\end{split}
\]
and hence, by \eqref{cd:cs 1}, \eqref{eq:near p^+}, \eqref{eq:near p^-}, 
\eqref{eq:competition} and $\ep_{\pm}\in[0,\delta/2]$, 
\[
\begin{split}
z^+_{l_+,t}&-D_{l_+}\Delta z_{l_+}^+=f_{l_+}(U^{+})-f_{l_+}(W^{+,\sigma^*-\eta})\\
&\geq f_{l_+}(U^{+})-f_{l_+}(W^{+,\sigma^*-\eta} -\ep_{+}\varphi^+)
\geq -M z^+_{l_+}\ \ {\rm on}\ \ S_+,
\end{split}
\]
\[
\begin{split}
z^-_{l_-,t}&-D_{l_-}\Delta z^-_{l_-}=f_{l_-}(U^{-})-f_{l_-}(W^{-,\sigma^*-\eta})\\
&\geq f_{l_-}(U^{-} +\ep_{-}\varphi^-)-f_{l_-}(W^{-,\sigma^*-\eta})
\geq -M z^{-}_{l_-}\ \ {\rm on}\ \ S_-.
\end{split}
\]
By strong maximum principle,
\[
z^{\pm}_{l_{\pm}}(x,t)=0\ \ 
{\rm for\ all}\ \ (x,t)\in S_{\pm}\cap (\R^N\times(-\infty,0])
\]
and this contradicts 
$z^{\pm}\succeq\epsilon_{\pm}\varphi^{\pm}\gg(0,0,\cdots,0)$ on $\partial S_{\pm}$. 
Thus $\ep_{\pm}=0$ and hence 
\[
u\succeq w^{\sigma^*-\eta}\ \ {\rm on}\ \ S_{\pm}\ \ 
{\rm for\ any}\ \ \eta\in[0,\eta_1].
\]
Therefore, by \eqref{eq:inequality}, 
it holds that 
$u\succeq w^{\sigma^*-\eta}\ 
{\rm for\ any}\ \eta\in[0,\eta_1].$ 
This contradicts the minimality of $\sigma^*$. 
Thus $\sigma^*\leq0$ and hence 
\[
u(x,t)\succeq w^0(x,t)=u(x_1+c\tau,x'+\rho,t+\tau).
\]
Since $(\rho,T)\in\R^{N-1}\times\R$ is arbitrary, there is a function $\widetilde{\phi}$ 
such that 
\[
u(x,t)=\widetilde{\phi}(x_1-ct)\ \ {\rm with}\ \ 
\widetilde{\phi}(-\infty)=p^+,\ \widetilde{\phi}(+\infty)=p^-.
\] 
Moreover $\widetilde{\phi}'\preceq(0,0,\cdots,0)$ 
since $\widetilde{\phi}(x_1-ct)\succeq \widetilde{\phi}(x_1-ct+\sigma)$ for all $\sigma>0$.
By strong maximum principle and 
$\widetilde{\phi}(-\infty)=p^+\gg p^-=\widetilde{\phi}(+\infty)$, 
\[
\widetilde{\phi}'\ll(0,0,\cdots,0).
\] 
If, in addition, assume \eqref{cd:irreducible 1}, then,  
by Lemma \ref{lm:uniqueness}, 
\[
\widetilde{\phi}(\cdot)=\phi(\cdot-\theta_0).
\] 
Then 
$\theta_0\in(a,b)$ follows from $\phi(\cdot-a)\preceq \phi(\cdot-\theta_0)
\preceq \phi(\cdot-b)$ and monotonicity of $\phi$. 
\end{proof}

\begin{proof}[Proof of Lemma \ref{lm:uniqueness}]
Define 
\[
\tau_0:=\inf\{\tau'>0\mid
\exists\tau\in\R,\ 
\widetilde{\phi}(\cdot-\tau)\preceq\phi(\cdot)\preceq\widetilde{\phi}(\cdot-\tau-\tau')\}
\ (\in[0,b-a])
\]
and we prove $\tau_0=0$ by contradiction. 
Suppose $\tau_0>0$. Then there are $\tau'_j,\,\tau_j\in\R$ such that 
$\tau_j'\rightarrow\tau_0$ as $j\rightarrow\infty$,
\[ 
\widetilde{\phi}(\cdot-\tau_j)\preceq\phi(\cdot)\preceq\widetilde{\phi}(\cdot-\tau_j-\tau'_j)
\ \ {\rm for}\ \ j=1,2,\cdots.
\]
By \eqref{eq:sandwiched} and monotonicity of $\widetilde{\phi}$, 
\[
\tau_j\in[a,b].
\] 
Hence, by extracting a subsequence, we may assume that $\tau_j$ converges to 
a $\tau_*$. Then 
\[
\widetilde{\phi}(\cdot-\tau_*)\preceq\phi(\cdot)
\preceq\widetilde{\phi}(\cdot-\tau_*-\tau_0).
\]
Let us take $\delta>0$, $\ep>0$, 
$C>\max\{|a|,|b|\}$ as in the proof of Theorem \ref{th:Ltype 1}. 
By Proposition \ref{pr:scp}, 
\[
\widetilde{\phi}(n\cdot x-ct-\tau_*)\ll\phi(n\cdot x-ct)\  
((x,t)\in\R^N\times\R).
\] 
Hence 
\begin{equation}\label{eq:strictly positivity}
(0,0,\cdots,0)\ll\inf\{\phi(s)-\widetilde{\phi}(s-\tau_*)\mid
-2C\leq s\leq2C\}.
\end{equation}
As in the proof of Theorem \ref{th:Ltype 1}, 
for any sufficiently small $\eta>0$, 
\[
\widetilde{\phi}(\cdot-\tau_*-\eta)\preceq \phi(\cdot).
\]
This implies 
\[
\tau_0
:=\inf\{\tau'>0\mid
\exists\tau\in\R,\ 
\widetilde{\phi}(\cdot-\tau)\preceq\phi(\cdot)\preceq\widetilde{\phi}(\cdot-\tau-\tau')\}
\leq\tau_0-\eta<\tau_0.
\]
This is contradiction and $\tau_0$ is equal to $0$. 
Therefore there is a $\theta_0\in\R$ such that
$\widetilde{\phi}(\cdot-\theta_0)=\phi(\cdot).$
\end{proof} 

%%%%%%%%%%%%%%%%%%%%%%%%%%%%%%%%%%%%%%%%%
\subsection{Outline of the proof of Theorems \ref{th:Ltype 2}\ and \ref{th:Ltype 3}}
The proof of the following two propositions is 
same as that of Proposition \ref{pr:scp} 
and we omit the proof. 
\begin{proposition}[strong comparison principle]\label{pr:t-scp}
Assume \eqref{cd:ellipticity}, \eqref{cd:T-periodic}, \eqref{cd:t-bistable} and \eqref{cd:t-cs}. 
Let $u(x,t)$, $v(x,t)$ be solutions of \eqref{eq:t-crds} such that 
\[
p^-\preceq u,\,v\preceq p^+,\ \ 
u(\cdot,0)\preceq v(\cdot,0).
\]
Then 
$u(\cdot,t)\preceq v(\cdot,t)$ for any $t\geq0$. 
If, in addition, assume \eqref{cd:t-irreducible} and $u(\cdot,0)\prec v(\cdot,0)$, then 
$u(\cdot,t)\ll v(\cdot,t)$ for any $t>0$. 
\end{proposition}

\begin{proposition}[strong comparison principle]\label{pr:x-scp}
Assume \eqref{cd:ellipticity 1}, \eqref{cd:L-periodic}, \eqref{cd:x-bistable} and \eqref{cd:x-cs}. 
Let $u(x,t)$, $v(x,t)$ be solutions of \eqref{eq:x-crds} such that 
\[
p^-\preceq u,\,v\preceq p^+,\ \ 
u(\cdot,0)\preceq v(\cdot,0).
\]
Then 
$u(\cdot,t)\preceq v(\cdot,t)$ for any $t\geq0$. 
If, in addition, assume \eqref{cd:x-irreducible} and $u(\cdot,0)\prec v(\cdot,0)$, then 
$u(\cdot,t)\ll v(\cdot,t)$ for any $t>0$.
\end{proposition}

The following two lemmas play key rules to prove 
the last parts of Theorems \ref{th:Ltype 2} and \ref{th:Ltype 3}, respectively. 
\begin{lemma}\label{lm:t-uniqueness}
Assume \eqref{cd:ellipticity}, \eqref{cd:T-periodic}, 
\eqref{cd:t-bistable}, \eqref{cd:t-cs} and \eqref{cd:t-irreducible}. 
Let 
\[
\phi(z,t),\ \widetilde{\phi}(z,t)\ \ (z=n\cdot x-ct)
\] 
be functions satisfying $({\rm A1})$ with a direction $n\in\R^N$ and a speed $c$ 
and for some constants $a,b\in\R$ 
and for all $z\in\R,\ t\in\R$, 
\begin{equation}\label{eq:t-sandwiched}
\widetilde{\phi}(z-a,t)\preceq\phi(z,t)
\preceq\widetilde{\phi}(z-b,t).
\end{equation}
Then $\phi(z,t)\equiv\widetilde{\phi}(z-\theta_0,t)$ 
for some $\theta_0\in\R$. 
\end{lemma}

\begin{lemma}\label{lm:x-uniqueness}
Assume \eqref{cd:x-bistable}, \eqref{cd:x-cs} and \eqref{cd:x-irreducible}. 
Let 
\[
u(x,t),\ v(x,t)\ \ ((x,t)\in\R^N\times\R)
\] 
be functions satisfying $({\rm A2})$ with a direction $n\in\R^N$ and a speed $c\not=0$ 
and for some constants $a,b\in\R$ and for all $x\in\R^N,\ t\in\R$, 
\begin{equation}\label{eq:x-sandwiched}
u(x,t+a)\preceq v(x,t)\preceq u(x,t+b).
\end{equation} 
Then 
$v(x,t)\equiv u(x,t-\theta_0)$ 
for some 
$\theta_0\in\R$. 
\end{lemma}

\begin{proof}[Outline of the proof of Theorem \ref{th:Ltype 2}]
Take 
\[
0<\delta<
\min\Big\{\frac{p^+_l(t)-p^-_l(t)}{\max\{\varphi^+_l(t),\varphi^-_l(t)\}}
\mid t\in\R,\ l=1,2,\cdots,m\Big\}
\] 
such that 
\begin{equation}\label{eq:t-delta}
\delta<\min\Big\{\frac{\lambda_+\varphi^+_l(t)}{a_+(t)},
\frac{\lambda_-\varphi^-_l(t)}{a_-(t)}\mid
t\in\R,\ l=1,2,\cdots,m\Big\},
\end{equation}
where $a_{\pm}(t):=
\max\Big\{1,\underset{w\in[p^-(t),p^+(t)]}{\sup}(|D^2F(t,w)||\varphi^{\pm}(t)|^2)\Big\}$, 
\[
|D^2F|:=\sqrt{\sum_{l=1}^m\sum_{i,j=1}^N f_{l,u_i u_j}^2},\ \ 
|\varphi|^2:=\sum_{l=1}^m\varphi_l^2.
\]
Then, by the same calculation as in the proof of Theorem \ref{th:Ltype 1}, 
\begin{equation}\label{eq:t-near p^+}
\left\{
\begin{split}
&F(t,w)-\ep DF(t,p^+(t))\varphi^+(t)-\ep\lambda_+\varphi^+(t)
\preceq F(t,w-\ep\varphi^+(t))\\
&{\rm for\ any}\ \ w\in\R^m,\ t\in\R\ \ {\rm with}\ \ 
p^+(t)\succeq w\succeq p^+(t)-\ep\varphi^+(t)\ \ {\rm and}\\
&{\rm for\ any}\ \ \ep\in[0,\delta/2],
\end{split}
\right.
\end{equation}
\begin{equation}\label{eq:t-near p^-}
\left\{
\begin{split}
&F(t,w)+\ep DF(t,p^-(t))\varphi^-(t)+\ep\lambda_-\varphi^-(t)
\succeq F(t,w+\ep\varphi^-(t))\\
&{\rm for\ any}\ \ w\in\R^m,\ t\in\R\ \ {\rm with}\ \ 
p^-(t)\preceq w\preceq p^-(t)+\ep\varphi^-(t)\ \ {\rm and}\\
&{\rm for\ any}\ \ \ep\in[0,\delta/2],
\end{split}
\right.
\end{equation}
We also take $C>b-a$ such that 
\begin{equation}
\left\{
\begin{split}
p^-(t)&\preceq u(x,t)\preceq p^-(t)+\frac{\delta}{2}\varphi^-(t)\\
&\ \ {\rm for\ all}\ \ (x,t)\in\R^N\times\R\ \ {\rm with}\ \ n\cdot x-ct\geq C,\\
p^+(t)&\succeq u(x,t)\succeq p^+(t)-\frac{\delta}{2}\varphi^+(t)\\
&\ \ {\rm for\ all}\ \ (x,t)\in\R^N\times\R\ \ {\rm with}\ \ n\cdot x-ct\leq -C.
\end{split}
\right.
\end{equation}
For any $(\rho,\tau)\in\R^N\times T\Z$ with $n\cdot\rho-c\tau=0$, 
an argument similar to that in the proof of Theorem \ref{th:Ltype 1} shows that, for any 
$\sigma\geq0$, 
\[
w^{\sigma}(x,t):=u(x+\rho+\sigma n,t+\tau)\preceq u(x,t)\ \ 
{\rm for\ all}\ \ x\in\R^N,\ t\in\R.
\]
This implies that $u(x,t)=\widetilde{\phi}(z,t)$ $(z=n\cdot x-ct)$ for 
a function $\widetilde{\phi}$ which satisfies 
\[
\widetilde{\phi}(z,t+T)\equiv\widetilde{\phi}(z,t),\ 
\widetilde{\phi}_z\gg(0,0,\cdots,0),\ 
\widetilde{\phi}(\pm\infty,\cdot)=p^{\mp}(\cdot).
\]
Moreover \eqref{cd:t-irreducible} and Lemma \ref{lm:t-uniqueness} imply 
$\widetilde{\phi}(z,t)\equiv\phi(z-\theta_0,t)$. 
\end{proof}

\begin{proof}[Outline of the proof of Theorem \ref{th:Ltype 3}]
We assume that the speed $c$ is positive 
since the sign of the speed is irrelevant in the later argument. 
Take 
\[
0<\delta<
\min\Big\{\frac{p^+_l(x)-p^-_l(x)}{\max\{\varphi^+_l(x),\varphi^-_l(x)\}}
\mid x\in\R^N,\ l=1,2,\cdots,m\Big\}
\] 
such that 
\begin{equation}\label{eq:x-delta}
\delta<\min\Big\{\frac{\lambda_+\varphi^+_l(x)}{a_+(x)},
\frac{\lambda_-\varphi^-_l(x)}{a_-(x)}\mid
x\in\R^N,\ l=1,2,\cdots,m\Big\},
\end{equation}
where $a_{\pm}(x):=
\max\Big\{1,\underset{w\in[p^-(x),p^+(x)]}{\sup}(|D^2F(x,w)||\varphi^{\pm}(x)|^2)\Big\}$, 
\[
|D^2F|:=\sqrt{\sum_{l=1}^m\sum_{i,j=1}^N f_{l,u_i u_j}^2},\ \ 
|\varphi|^2:=\sum_{l=1}^m\varphi^2.
\]
Then, by the same calculation as in the proof of Theorem \ref{th:Ltype 1}, 
\begin{equation}\label{eq:x-near p^+}
\left\{
\begin{split}
&F(x,w)-\ep DF(x,p^+(x))\varphi^+(x)-\ep\lambda_+\varphi^+(x)
\preceq F(x,w-\ep\varphi^+(x))\\
&{\rm for\ any}\ \ w\in\R^m,\ x\in\R^N\ \ {\rm with}\ \ 
p^+(x)\succeq w\succeq p^+(x)-\ep\varphi^+(x)\ \ {\rm and}\\
&{\rm for\ any}\ \ \ep\in[0,\delta/2],
\end{split}
\right.
\end{equation}
\begin{equation}\label{eq:x-near p^-}
\left\{
\begin{split}
&F(x,w)+\ep DF(x,p^-(x))\varphi^-(x)+\ep\lambda_-\varphi^-(x)
\succeq F(x,w+\ep\varphi^-(x))\\
&{\rm for\ any}\ \ w\in\R^m,\ x\in\R^N\ \ {\rm with}\ \ 
p^-(x)\preceq w\preceq p^-(x)+\ep\varphi^-(x)\ \ {\rm and}\\
&{\rm for\ any}\ \ \ep\in[0,\delta/2],
\end{split}
\right.
\end{equation}
We also take $C>b-a$ such that 
\begin{equation}
\left\{
\begin{split}
p^-(x)&\preceq u(x,t)\preceq p^-(x)+\frac{\delta}{2}\varphi^-(x)\\
&\ \ {\rm for\ all}\ \ (x,t)\in\R^N\times\R\ \ {\rm with}\ \ n\cdot x-ct\geq C,\\
p^+(x)&\succeq u(x,t)\succeq p^+(x)-\frac{\delta}{2}\varphi^+(x)\\
&\ \ {\rm for\ all}\ \ (x,t)\in\R^N\times\R\ \ {\rm with}\ \ n\cdot x-ct\leq -C.
\end{split}
\right.
\end{equation}
For any $(\rho,\tau)\in\mathbb{L}\times \R$ with $n\cdot\rho-c\tau=0$, 
an argument similar to that in the proof of Theorem \ref{th:Ltype 1} shows that, 
for any $\sigma\geq0$,  
\[
w^{\sigma}(x,t):=u(x+\rho,t+\tau-\sigma)\preceq u(x,t)\ \ 
{\rm for\ all}\ \ x\in\R^N,\ y\in\R.
\]
This implies that, for $\rho\in\mathbb{L}$, $t\in\R$, 
\[
u(x+\rho,t+n\cdot\rho/c)\equiv u(x,t),\ 
u_t\succeq(0,0,\cdots,0).
\] 
By \eqref{eq:sandwich 1} and 
$\underset{k\in\mathbb{L},n\cdot k\rightarrow\pm\infty}{\lim}
v(\cdot+k,t)=p^{\mp}(\cdot)$, 
\begin{equation}\label{eq:asym beha}
\underset{k\in\mathbb{L},n\cdot k\rightarrow\pm\infty}{\lim}
u(\cdot+k,t)=p^{\mp}(\cdot).
\end{equation}
From \eqref{eq:asym beha}, maximum principle, $u_t\succeq(0,0,\cdots,0)$ and 
\[
\begin{split}
u_{l,tt}\geq&\sum_{i,j=1}^N D_l^{ij}(x)u_{l,t x_ix_j}+q_l(x)\cdot\nabla u_{l,t}+f_{l,u_l}(x,u_1,\cdots,u_m)u_{l,t}\\
&{\rm for }\ \ x\in\R^N,\ t\in\R\ \ (l=1,2,\cdots,m),
\end{split}
\]
it holds that 
$u_t\gg(0,0,\cdots,0).$ 
Therefore $u$ is a solution which satisfies $({\rm A2})$. 
Moreover, if, in addition, assume \eqref{cd:x-irreducible}, then Lemma \ref{lm:x-uniqueness} implies 
$u(x,t)\equiv v(x,t+\theta_0)$. 
\end{proof}

\begin{proof}
[Proof of Lemmas \ref{lm:t-uniqueness} and \ref{lm:x-uniqueness}]
The proof of Lemmas \ref{lm:t-uniqueness} and \ref{lm:x-uniqueness} 
is based on Propositions \ref{pr:t-scp}, \ref{pr:x-scp} 
and an argument similar to 
that in the proof of Lemma \ref{lm:uniqueness}. 
We give the proof of Lemma \ref{lm:x-uniqueness} only. 
The proof of Lemma \ref{lm:t-uniqueness} is easier and omitted. 
We only consider the case that the peed $c$ is positive since the sign of the speed is 
irrelevant in the later argument. 

Define 
\[
\begin{split}
\tau_0:=\{\tau'\mid \exists \tau\in\R,\ 
&u(x,t+\tau)\preceq v(x,t)\preceq u(x,t+\tau+\tau')
\\ &((x,t)\in\R^N\times\R)\}\ 
(\in[0,b-a])
\end{split}
\]
and we prove $\tau_0=0$ by contradiction. Suppose $\tau_0>0$. Then there are 
$\tau_j'$, $\tau_j\in\R$ such that $\tau_j'\rightarrow\tau_0$ as $j\rightarrow\infty$, 
\[
u(x,t+\tau_j)\preceq v(x,t)\preceq u(x,t+\tau_j+\tau_j')
\ \ {\rm for}\ \ (x,t)\in\R^N\times\R\ \ (j=1,2,\cdots).
\]
By \eqref{eq:x-sandwiched} and monotonicity of $u$, $v$ with respect to $t$, 
\[
\tau_j\in[a,b].
\] 
Hence, by extracting a subsequence, 
we may assume that $\tau_j$ converges to a $\tau_*$ as $j\rightarrow\infty$. 
Then 
\[
u(x,t+\tau_*)\preceq v(x,t)\preceq u(x,t+\tau_*+\tau_0)\ \ 
((x,t)\in\R^N\times\R).
\]
Let us take $\delta>0$, $\ep>0$, $C>\max\{|a|,|b|\}$ 
as in the proof of Theorem \ref{th:Ltype 3}. By Proposition \ref{pr:x-scp}, 
\[
u(x,t+\tau_*)\ll v(x,t)\ \ ((x,t)\in\R^N\times\R)
\] 
and hence 
\begin{equation}\label{eq:x-strictly positivity}
(0,0,\cdots,0)\ll\inf\{v(x,t)-u(x,t+\tau_*)\mid (x,t)\in\R^N\times\R,\ |n\cdot x-ct|\leq2C\}.
\end{equation}
If \eqref{eq:x-strictly positivity} is not true, then there are
$l_0\in\{1,2,\cdots,m\}$, $x_j\in\R^N$, $t_j\in\R$ 
such that 
\[
|n\cdot x_j-ct_j|\leq2C\ \ (j\in\N)\ \ {\rm and}\ \  
\lim_{j\rightarrow\infty}\{v_{l_0}(x_j,t_j)-u_{l_0}(x_j,t_j+\tau_*)\}=0.
\]
Let $k_j\in\mathbb{L}\,(=L_1\Z\times L_2\Z\times\cdots\times L_N\Z)$ satisfy 
\[
x_j\in k_j+[0,L_1)\times[0,L_2)\times\cdots\times[0,L_N)\ \ (j\in\N).
\]
Then  $x_j-k_j$, $t_j-n\cdot k_j/c$ are bounded uniformly for $j\in\N$. 
Hence, by extracting a subsequence, we may assume that there are 
$x_*\in\R^N$ and $t_*\in\R$ such that $|x_*\cdot n-ct_*|\leq 2C$, 
$
x_j-k_j\rightarrow x_*,\ 
t_j-n\cdot k_j/c\rightarrow t_*\ \ {\rm as}\ \ j\rightarrow\infty.
$
Thus 
\[
\begin{split}
0&=\lim_{j\rightarrow\infty}\{v_{l_0}(x_j,t_j)-u_{l_0}(x_j,t_j+\tau_*)\}\\
&=\lim_{j\rightarrow\infty}\{v_{l_0}(x_j-k_j,t_j-n\cdot k_j/c)-u_{l_0}(x_j-k_j,t_j-n\cdot k_j/c+\tau_*)\}\\
&=v_{l_0}(x_*,t_*)-u_{l_0}(x_*,t_*+\tau_*).
\end{split}
\]
This contradicts $u(x,t+\tau_*)\ll v(x,t)$\ \ $((x,t)\in\R^N\times\R)$ and 
\eqref{eq:x-strictly positivity} holds. By an argument similar to that in the proof 
of Theorem \ref{th:Ltype 1}, for any sufficiently small $\eta>0$, 
$
u(x_*,t_*+\tau_*+\eta)\preceq v(x_*,t_*).
$
Thus 
\[
\tau_0=\inf\{\tau'\mid \exists \tau\in\R,\ 
u(\cdot,\cdot+\tau)\preceq v(\cdot,\cdot)\preceq u(\cdot,\cdot+\tau+\tau')\}
\leq \tau_0-\eta<\tau_0.
\]
This is contradiction and $\tau_0=0$ is proved. 
This completes the proof. 
\end{proof}
%%%%%%%%%%%%%%%%%%%%%%%%%%%%%%%%%%%%%%%%%
\section*{\acknowledgments}
The author would like to thank Prof. Matano for many helpful suggestions and 
continuous encouragement. 
%%%%%%%%%%%%%%%%%%%%%%%%%%%%%%%%%%%%%%%%%
%%%%%%%%%%%%%%%%%%%%%%%%%%%%%%%%%%%%%%%%%%%%%%%%%%%%%%%%%%%%%%%%%%
\bibliographystyle{amsplain}

\end{document}